\documentclass[11pt]{article}
\usepackage{epigamath}

\usepackage[notext]{kpfonts}
\usepackage{baskervald}

\setpapertype{A4}

\usepackage[english]{babel}

\usepackage[dvips]{graphicx}     

\title{\vspace{-1cm}Smooth affine group schemes over the dual numbers}
\titlemark{Smooth affine group schemes over the dual numbers}
\author{\vspace{0cm} Matthieu Romagny and Dajano Tossici}
\authoraddresses{
\authordata{Matthieu Romagny}{\firstname{Matthieu} \lastname{Romagny}\\
\institution{Universit\'e de Rennes, CNRS, IRMAR -- UMR 6625, F-35000 Rennes, France}\\
\email{matthieu.romagny@univ-rennes1.fr}}\\
\authordata{Dajano Tossici}{\firstname{Dajano} \lastname{Tossici}\\
\institution{Institut de Math\'ematiques de Bordeaux,
351 Cours de la Liberation, 33405 Talence, France}\\
\email{Dajano.Tossici@math.u-bordeaux.fr}}
}
\authormark{M. Romagny and D. Tossici}
\date{\vspace{-5ex}} 
\journal{\'Epijournal de G\'eom\'etrie Alg\'ebrique} 
\acceptation{Received by the Editors on August 30, 2018, and in final form
on December 28, 2018. \\ Accepted on March 4, 2019.}

\newcommand{\articlereference}{Volume 3 (2019), Article Nr. 5}

\usepackage[all]{xy}

\allowdisplaybreaks

\makeatletter
\renewcommand\thesubsection{{\thesection}\@arabic\c@subsection.}
\renewcommand{\p@subsection}{\arabic{section}.\arabic{subsection}\expandafter\@gobble}
\makeatother

\newcounter{subsubsectionnum}
\numberwithin{subsubsectionnum}{subsection}
\renewcommand{\thesubsubsectionnum}{{\thesubsection}\arabic{subsubsectionnum}}
\newlength{\lengthtitle}
\newcommand{\newsubsubsection}[1]{
\settowidth{\lengthtitle}{#1}
\ifnum\lengthtitle=0\paragraph{\bfseries \thesubsubsectionnum.}\else\paragraph{\bfseries \thesubsubsectionnum.~#1.}\fi
\refstepcounter{subsubsectionnum}}

\newcounter{subsubsectionnumB}
\renewcommand{\thesubsubsectionnumB}{B.\arabic{subsubsectionnumB}}
\newcommand{\newsubsubsectionB}[1]{
\settowidth{\lengthtitle}{#1}
\ifnum\lengthtitle=0\paragraph{\bfseries \thesubsubsectionnumB.}\else\paragraph{\bfseries \thesubsubsectionnumB.~#1.}\fi
\refstepcounter{subsubsectionnumB}}

\newenvironment{item-title}[1]{\newsubsubsection{#1}}{}
\newenvironment{item-titleB}[1]{\newsubsubsectionB{#1}}{}

\newenvironment{remark}{\begin{item-title}{Remark}}{\end{item-title}}

\newenvironment{definition}{\begin{item-title}{Definition}}{\end{item-title}}
\newenvironment{prop}{\begin{item-title}{Proposition}\em}{\rm\end{item-title}}
\newenvironment{example-title}[2]{\begin{item-title}{Example~#1: #2}}{\end{item-title}}
\newenvironment{theocite}[1]{\begin{item-title}{Theorem~(#1)}\em}{\rm\end{item-title}}
\newenvironment{lemma}{\begin{item-title}{Lemma}\em}{\rm\end{item-title}}
\newenvironment{corollary}{\begin{item-title}{Corollary}\em}{\rm\end{item-title}}
\newenvironment{example}{\begin{item-title}{Example}}{\end{item-title}}
\newenvironment{theorem}{\begin{item-title}{Theorem}\em}{\rm\end{item-title}}

\newenvironment{proof-of}[1]{\begin{proof}[Proof of #1]}{\end{proof}}

\newenvironment{definition*}{\begin{item-titleB}{Definition}}{\end{item-titleB}}

\newcommand{\smallvee}{{\scriptscriptstyle\vee}}

\newcommand{\eps}{\varepsilon}

\renewcommand{\ge}{\geqslant}

\def\itemn#1{\item[\hspace{0.6mm} {\rm (#1)}]}
\def\itemm#1{\item[\hspace{10mm} {\rm (#1)}]}

\def\too{\longrightarrow}
\def\tooo{\relbar\joinrel\relbar\joinrel\longrightarrow}

\def\into{\hookrightarrow}
\newcommand*{\intoo}{\ensuremath{\lhook\joinrel\relbar\joinrel\rightarrow}}

\def\isomto{\xrightarrow{\,\smash{\raisebox{-0.5ex}{\ensuremath{\scriptstyle\sim}}}\,}}
\renewcommand{\tilde}{\widetilde}

\mathchardef\ordinarycolon\mathcode`\:
\mathcode`\:=\string"8000
\begingroup \catcode`\:=\active
  \gdef:{\mathrel{\mathop\ordinarycolon}}
\endgroup

\newtheorem{theointro}{Theorem}

\DeclareMathOperator{\Aut}{Aut}
\DeclareMathOperator{\Hom}{Hom}
\DeclareMathOperator{\End}{End}
\DeclareMathOperator{\GL}{GL}
\DeclareMathOperator{\Ext}{Ext}
\DeclareMathOperator{\ext}{ext}
\DeclareMathOperator{\Sch}{Sch}
\DeclareMathOperator{\Fun}{Fun}
\DeclareMathOperator{\Gr}{Gr}
\DeclareMathOperator{\id}{id}

\DeclareMathOperator{\Spec}{Spec}

\DeclareMathOperator{\Mod}{Mod}
\DeclareMathOperator{\Der}{Der}
\DeclareMathOperator{\Alg}{Alg}

\DeclareMathOperator{\Lie}{Lie}
\DeclareMathOperator{\Sym}{S}
\DeclareMathOperator{\diff}{d}
\DeclareMathOperator{\Ad}{Ad}
\DeclareMathOperator{\Fr}{F}
\DeclareMathOperator{\V}{V}
\DeclareMathOperator{\CU}{CU}
\DeclareMathOperator{\SCU}{SCU}
\DeclareMathOperator{\ad}{ad}
\DeclareMathOperator{\T}{T}
\DeclareMathOperator{\Sets}{Sets}

\DeclareMathOperator{\Homc}{Homc}

 \def\cE{{\mathcal E}}

  \def\cO{{\mathcal O}}
  
\def\cS{{\mathcal S}}

\renewcommand\AA{\mathbb{A}} 
 \newcommand\DD{\mathbb{D}}
 \newcommand\FF{\mathbb{F}}
\newcommand\GG{\mathbb{G}} 
 
 \newcommand\LL{\mathbb{L}}
 
\newcommand\OO{\mathbb{O}}

 \newcommand\VV{\mathbb{V}}
 
 \newcommand\ZZ{\mathbb{Z}}

\begin{document}


\maketitle



\begin{prelims}

\vspace{-0.55cm}

\def\abstractname{Abstract}
\abstract{In this article we provide an equivalence between the category of
smooth affine group schemes over the ring of generalized dual
numbers $k[I]$, and the category of extensions of the form
$1\to \Lie(G,I)\to E\to G\to 1$ where $G$ is a smooth affine
group scheme over $k$. Here $k$ is an arbitrary commutative ring
and $k[I]=k\oplus I$ with $I^2=0$. The equivalence is given by
Weil restriction, and we provide a quasi-inverse which we call
{\em Weil extension}. It is compatible with the exact structures
and the $\OO_k$-module stack structures on both categories.
Our constructions rely on the use of the group algebra scheme
of an affine group scheme; we introduce this object and
establish its main properties. As an application, we establish
a Dieudonn\'e classification for smooth, commutative, unipotent
group schemes over $k[I]$ when $k$ is a perfect field.}



\vspace{0.15cm}

\languagesection{Fran\c{c}ais}{%

\textbf{Titre. Sch\'emas en groupes affines et lisses sur les nombres duaux} \commentskip \textbf{R\'esum\'e.} Nous construisons une \'equivalence entre la cat\'egorie des
sch\'emas en groupes affines et lisses sur l'anneau des nombres duaux
g\'en\'eralis\'es $k[I]$, et la cat\'egorie des extensions de la forme
$1\to \Lie(G,I)\to E\to G\to 1$ o\`u $G$ est un sch\'ema en groupes
affine, lisse sur $k$. Ici $k$ est un anneau commutatif arbitraire
et $k[I]=k\oplus I$ avec $I^2=0$. L'\'equivalence est donn\'ee par
la restriction de Weil, et nous construisons un foncteur quasi-inverse
explicite que nous appelons {\em extension de Weil}. Ces foncteurs sont
compatibles avec les structures exactes et avec les structures de champs
en $\OO_k$-modules des deux cat\'egories. Nos constructions s'appuient
sur le sch\'ema en alg\`ebres de groupe d'un sch\'ema en groupes affines,
que nous introduisons et dont nous donnons les propri\'et\'es principales.
En application, nous donnons une classification de Dieudonn\'e pour les
sch\'emas en groupes commutatifs, lisses, unipotents sur $k[I]$
lorsque $k$ est un corps parfait.}

\end{prelims}


\newpage

\setcounter{tocdepth}{2} \tableofcontents

\setcounter{section}{0}
\section{Introduction} \label{section:introduction}

Throughout this article, we fix a commutative ring $k$ and
a free $k$-module $I$ of finite rank $r\ge 1$. We consider
the ring of (generalized) dual numbers $k[I]:=k\oplus I$
with $I^2=0$.
We write $h:\Spec(k[I])\to\Spec(k)$ the structure map and
$i:\Spec(k)\to\Spec(k[I])$ the closed immersion. Also we denote
by $\OO_k$ the $\Spec(k)$-ring scheme such that if $R$ is a
$k$-algebra then $\OO_k(R)=R$ with its ring structure.

\subsection{Motivation, results, plan of the article}

\begin{item-title}{Motivation}
The starting point of our work is a relation between deformations
and group extensions. To explain the idea, let $G$ be an affine,
flat, finitely presented
$k$-group scheme. It is shown in Illusie's book~\cite{Il72}
that the set of isomorphism classes of deformations
of $G$ over $k[I]$ is in bijection with the cohomology group
$\mathrm{H}^2(BG,\underline{\ell}{}_G^\smallvee\otimes I)$,
see chap.~VII, thm~3.2.1 in {\it loc. cit}. Here, the
coefficients of cohomology are the derived dual of the
equivariant cotangent complex $\underline{\ell}{}_G \in D(BG)$,
tensored (in the derived sense) by $I$ viewed as the coherent
sheaf it defines on the fpqc site of $BG$, also equal to
the vector bundle $\VV(I^\smallvee)$.
If we assume moreover that $G$ is smooth, then
the augmentation $\underline{\ell}{}_G\to \omega^1_G$
to the sheaf of invariant differential 1-forms is a
quasi-isomorphism and it follows that
$\underline{\ell}{}_G^\smallvee\simeq \Lie G$.
Since coherent cohomology of $BG$ is isomorphic to group
cohomology of~$G$, the cohomology group of interest ends up
being $\mathrm{H}^2(G,\Lie(G,I))$ where
$\Lie(G,I):=\Lie G\otimes\VV(I^\smallvee)$.
The latter cohomology group is meaningful also in the theory
of group extensions, where it is known to classify isomorphism
classes of extensions of~$G$ by $\Lie(G,I)$ viewed as a
$G$-module via the adjoint representation, see
\cite[chap.~II, \S~3, no~3.3 and III, \S~6, no~2.1]{DG70}.

In this paper, our aim is to give a direct algebro-geometric
construction of this correspondence between deformations
and group extensions.
Our main result is that the Weil restriction
functor $h_*$ provides such a construction. Thereby, we obtain
a categorification of a link that has been available only as a
bijection between sets of isomorphism classes. This improvement
is crucial for a better understanding of $k[I]$-group schemes,
since in applications most groups occur as kernels or quotients
of morphisms. We illustrate this by giving a Dieudonn\'e-type
theory for unipotent group schemes. Natural extensions
of our result to more general thickenings
of the base, or to non-smooth group schemes, would have further
interesting applications. Since we wish to show our main results
to the reader without further delay, we postpone the discussion
of these applications to~\ref{further devpmts} below.
\end{item-title}

\begin{item-title}{Results} \label{item:results}
Our main result is an equivalence between the category of smooth,
affine $k[I]$-group schemes and a certain category of extensions
of $k$-group schemes. However, for reasons that are discussed below,
it is convenient for us to work with group schemes slightly more
general than smooth ones.
We say that a morphism of schemes $X\to S$ is
{\em differentially flat} if both $X$ and $\Omega^1_{X/S}$
are flat over $S$. Examples of differentially flat group schemes
include smooth group schemes, pullbacks from the spectrum
of a field, Tate-Oort group schemes with parameter $a=0$ in
characteristic~$p$, flat vector bundles i.e. $\VV(\mathscr{F})$
with~$\mathscr{F}$ flat over the base, and truncated Barsotti-Tate
groups of level~$n$ over a base where $p^n=0$
\cite[2.2.1]{Il85}. If $\mathscr{G}$ is a $k[I]$-group scheme, we call
{\em rigidification} of $\mathscr{G}$ an isomorphism of $k[I]$-schemes
$\sigma:h^*\mathscr{G}_k\isomto \mathscr{G}$ that lifts the identity of the
scheme $\mathscr{G}_k:=i^*\mathscr{G}$. Such a map need not be a morphism of
group schemes. We say that $\mathscr{G}$ is {\em rigid} if it admits
a rigidification. Examples of rigid group schemes include
smooth group schemes, pullbacks from $\Spec(k)$, and groups
of multiplicative type.  Finally a {\em deformation} of a
flat $k$-group scheme~$G$ over $k[I]$ is a pair composed of
a flat $k[I]$-group scheme $\mathscr{G}$ and an isomorphism
of $k$-group schemes $\mathscr{G}_k\isomto G$.

Let $\Gr\!/k[I]$ be the category of affine, differentially
flat, rigid $k[I]$-group schemes (this includes all smooth
affine $k[I]$-group schemes). The morphisms in this
category are the morphisms of $k[I]$-group schemes.
Let $\Ext(I)/k$ be the category of extensions of the form
$1\to \Lie(G,I) \to E \to G \to 1$ where~$G$ is an affine
differentially flat $k$-group scheme, and ``extension'' means
that $E\to G$ is an fpqc torsor under $\Lie(G,I)$. The
morphisms in this category are the commutative diagrams:
\[
\xymatrix{
1 \ar[r] & \Lie(G,I) \ar[r] \ar[d]^{\diff\!\psi}
& E \ar[r] \ar[d]^{\varphi} & G \ar[r] \ar[d]^{\psi} & 1\: \\
1 \ar[r] & \Lie(G',I) \ar[r] & E' \ar[r] & G' \ar[r] & 1}
\]
where $\diff\!\psi=\Lie(\psi)$ is the differential of $\psi$.
Usually such a morphism will be denoted simply
$\varphi:E\to E'$.

The categories $\Gr\!/k[I]$ and $\Ext(I)/k$ are exact categories.
They are also fpqc stacks over $\Spec(k)$ equipped with the
structure of {\em $\OO_k$-module stacks fibred in groupoids
over $\Gr/k$}. This means that there exist notions of {\em sum}
and {\em scalar multiple} for objects of $\Gr\!/k[I]$ and
$\Ext(I)/k$ (for extensions, the sum is the Baer sum); these
structures are described
in~\ref{subsection:the O-module stack structures}.
We can now state our main result.

\begin{theointro}[see~\ref{main_theorem}]
\label{main_theorem_0}
{\rm (1)} The Weil restriction functor provides an equivalence:
\[
h_*:\Gr\!/k[I] \isomto \Ext(I)/k.
\]
This equivalence commutes with base changes on $\Spec(k)$.

\smallskip

\noindent {\rm (2)} If $1\to\mathscr{G}'\to\mathscr{G}\to\mathscr{G}''\to 1$ is an
exact sequence in $\Gr\!/k[I]$, then
$1\to h_*\mathscr{G}'\to h_*\mathscr{G}\to h_*\mathscr{G}''$ is exact in $\Ext(I)/k$.
If moreover $\mathscr{G}'$ is smooth then
$1\to h_*\mathscr{G}'\to h_*\mathscr{G}\to h_*\mathscr{G}''\to 1$ is exact. In particular,
$h_*$ is an exact equivalence between the subcategories
of smooth objects endowed with their natural exact structure.

\smallskip

\noindent {\rm (3)} The equivalence $h_*$ is a morphism
of $\OO_k$-module stacks fibred over $\Gr\!/k$, i.e.
it transforms
the addition and scalar multiplication of deformations
of a fixed $G\in\Gr\!/k$ into the Baer sum and scalar
multiplication of extensions.

\smallskip

\noindent {\rm (4)} Let $P$ be one of the properties of
group schemes over a field: of finite type, smooth, connected,
unipotent, split unipotent, solvable, commutative. Say that
a group scheme over an arbitrary ring {\em has property~$P$} if
it is flat and its fibres have $P$. Then $\mathscr{G}\in \Gr\!/k[I]$
has property $P$ if and only if the $k$-group scheme $E=h_*\mathscr{G}$
has $P$.
\end{theointro}

In order to show that $h_*$ is an equivalence, we build
a quasi-inverse $h^{\mbox{\rm\tiny +}}$ which we call
{\em Weil extension}. The construction and study of this
functor is the hardest part of the proof.

As an application of the Theorem, we prove a Dieudonn\'e
classification for smooth, commutative, unipotent group
schemes over the generalized dual numbers of a perfect
field $k$. This takes the form of an
exact equivalence of categories with a category of
extensions of smooth, erasable Dieudonn\'e modules.
Here is the precise statement (we refer to
Section~\ref{section:application} for
the definition of all undefined terms).

\begin{theointro}[see~\ref{theorem:Dieudonne for SCU}] \label{main_theorem_1}
Let $\SCU/k[I]$ be the category of smooth, commutative,
unipotent $k[I]$-group schemes. Let
$\DD\mbox{-}I\mbox{-}\Mod$ be the category of $I$-extensions
of smooth erasable Dieudonn\'e modules. Then the Dieudonn\'e
functor $\underline M:\CU/k \too \DD\mbox{-}\Mod$
induces a contravariant equivalence of categories:
\[
\mathscr{M}:\SCU/k[I] \too \DD\mbox{-}I\mbox{-}\Mod
\]
that sends $\mathscr{U}$ to the extension
$0 \to \underline M(\mathscr{U}_k)\to \underline M(h_*\mathscr{U}))\to
\underline M(\Lie(\mathscr{U}_k,I))\to 0$.
\end{theointro}
\end{item-title}

\begin{item-title}{Comments}
An important tool in many of our arguments is the
{\em group algebra scheme}. It provides a common framework
to conduct computations in the groups and their tangent bundles
simultaneously. It allows us to describe conveniently the
Weil restriction of a group scheme, and is essential in the
proof of Theorem~\ref{main_theorem_0}.
Since we are not aware of any treatment of the group algebra
scheme in the literature, we give a detailed treatment
in Section~\ref{group algebra}. We point out that this concept
is useful in other situations; in particular it allows to work
out easily the deformation theory of smooth affine group schemes,
as we show in Subsection~\ref{subsection:defos}.

Let us say a word on the assumptions.
The choice to work with differentially flat group schemes
instead of simply smooth ones is not just motivated by the
search for maximal generality or aesthetic reasons. It is also
extremely useful because when working with an affine,
smooth group scheme $G$, we use our results also
for the group algebra $\OO_k[G]$ in the course of proving
the main theorem; and the group algebra
$\OO_k[G]$ is differentially flat and rigid, but usually
infinite-dimensional and hence not smooth.

There are at least two advantages to work over generalized dual
numbers $k[I]$ rather than simply the usual ring $k[\eps]$
with $\eps^2=0$. The first is that in order to prove that our
equivalence of categories respects the $\OO_k$-module stack
structure, we have to introduce the ring $k[I]$ with the
two-dimensional $k$-module $I=k\eps+k\eps'$. Indeed, this is
needed to describe the sum of deformations and the Baer sum
of extensions. The other advantage is that since arbitrary local
Artin $k$-algebras are filtered by Artin $k$-algebras whose
maximal ideal has square zero, our results may be useful in
handling deformations along more general thickenings.
\end{item-title}

\begin{item-title}{Further developments}
\label{further devpmts}
Our results have several desirable generalizations. Here are the two
most natural directions. First, one may wish to relax the assumptions
on the group schemes and consider non-affine or non-smooth group
schemes; second one may wish to consider more general thickenings
than that given by the dual numbers. Let us explain how our personal
work indicates a specific axis for research. In previous work
of the authors with Ariane M\'ezard~\cite{MRT13}, we studied models
of the group schemes of roots of unity $\mu_{p^n}$ over $p$-adic rings.
As a result, we raised a conjecture which says in essence that every
such model can be equipped with a cohomological theory that
generalizes the Kummer theory available on the generic fibre.
In the process of trying to
prove the conjecture, we encountered various character groups of
smooth and finite unipotent group schemes over truncated discrete
valuation rings. In order to compute these, it is therefore desirable
to obtain a statement similar to Theorem~\ref{main_theorem_0}
in this context. The present paper can be seen as
the first part of this programme, carried out in the simplest case;
we plan to realize the second part of the programme by using the
{\em derived Weil restriction} or {\em derived Greenberg functor}
in place of the usual Weil restriction.
\end{item-title}

\begin{item-title}{Plan of the article}
The present Section~\ref{section:introduction} ends with
material of preliminary nature on the description of the
$\OO_k$-module stack structure of the categories $\Gr\!/k[I]$
and $\Ext(I)/k$ and on Weil restriction.
In Section~\ref{group algebra} we introduce group
algebra schemes. In Section~\ref{section:Weil restriction} we
describe the functor $h_*:\Gr\!/k[I] \to \Ext(I)/k$, in
Section~\ref{section:functor tit} we construct a functor
$h^{\mbox{\rm\tiny +}}:\Ext(I)/k\to \Gr\!/k[I]$, while in
Section~\ref{section:eq of cats} we prove that these
functors are quasi-inverse and we complete the proof
of Theorem~\ref{main_theorem_0}. Finally, in
Section~\ref{section:application} we derive the Dieudonn\'e
classification for smooth commutative unipotent group schemes
over the dual numbers.  In the Appendices, we review notions
from differential calculus (tangent bundle, Lie algebra and
exponentials) and module categories (Picard categories with
scalar multiplication) in the level of generality
needed in the paper.
\end{item-title}

\begin{item-title}{Acknowledgements}
For various discussions and remarks, we thank Sylvain Brochard,
Xavier Caruso, Brian Conrad, Bernard Le Stum, Brian Osserman,
and Tobias Schmidt. We acknowledge the help of the referee
to make the article much more incisive. We are also grateful
to the CIRM in Luminy where part of this
research was done. Finally, the first author would like to thank
the executive and administrative staff of IRMAR and of the Centre
Henri Lebesgue ANR-11-LABX-0020-01 for creating an attractive mathematical environment.
\end{item-title}


\subsection{The $\OO_k$-module stack structure of
$\Gr\!/k[I]$ and $\Ext(I)/k$}
\label{subsection:the O-module stack structures}

Both categories $\Gr\!/k[I]$ and $\Ext(I)/k$ are endowed with
the structure of {\em $\OO_k$-module stacks in groupoids over
$\Gr\!/k$}. The reader who wishes to see the full-fledged
definition is invited to read
Appendix~\ref{appendix:module groupoids}.
In rough terms, once a $k$-group scheme $G$ is fixed, the
$\OO_k$-module category structure boils down to an addition law
by which one can add deformations (resp. extensions) of $G$,
and an external law by which one can multiply a deformation
(resp. an extension) by scalars of the ring scheme $\OO_k$.
Here is a description of these structures.

\begin{item-title}{The $\OO_k$-module stack
$\Gr\!/k[I]\to \Gr\!/k$}
Let $G\in \Gr\!/k$ be fixed.
Let $\mathscr{G}_1,\mathscr{G}_2\in\Gr\!/k[I]$ with identifications
$i^*\mathscr{G}_1\simeq G\simeq i^*\mathscr{G}_2$. The addition is obtained
by a two-step process. First we glue these group schemes along
their common closed subscheme $G$:
\[
\mathscr{G}':=\mathscr{G}_1\amalg_G\mathscr{G}_2.
\]
This lives over the scheme
$\Spec(k[I])\times_{\Spec(k)}\Spec(k[I])=\Spec(k[I\oplus I])$.
The properties of gluing along infinitesimal thickenings are
studied in the Stacks Project~\cite{SP}. We point out some
statements relevant to our situation:
existence of the coproduct in
\href{https://stacks.math.columbia.edu/tag/07RV}{Tag~07RV},
a list of properties preserved by gluing in
\href{https://stacks.math.columbia.edu/tag/07RX}{Tag~07RX},
gluing of modules in
\href{https://stacks.math.columbia.edu/tag/08KU}{Tag~08KU}
and preservation of flatness of modules in
\href{https://stacks.math.columbia.edu/tag/07RW}{Tag~07RW}.
It follows from these results that $\mathscr{G}'$ is an object of
$\Gr\!/k[I\oplus I]$. Then we form the desired sum
\[
\mathscr{G}_1+\mathscr{G}_2:=j^*(\mathscr{G}')=j^*(\mathscr{G}_1\amalg_G\mathscr{G}_2)
\]
by pullback along the closed immersion
\[
j:\Spec(k[I]) \intoo \Spec(k[I\oplus I])
\]
induced by the addition morphism $I\oplus I\to I$,
$i_1\oplus i_2\mapsto i_1+i_2$.
The neutral element for this addition is the group scheme
$h^*G$. The scalar multiple
\[
\lambda\mathscr{G}:=s_\lambda^*\mathscr{G}
\]
is given by rescaling using the scheme map
$s_\lambda:\Spec(k[I])\to\Spec(k[I])$ induced by the
$k$-algebra map $k[I]\to k[I]$ which is multiplication by
$\lambda$ in $I$. The axioms of $\OO_k$-module stacks can be
checked but we leave the details to the courageous reader.
\end{item-title}

\begin{item-title}{The $\OO_k$-module stack
$\Ext(I)/k\to \Gr\!/k$}
Again let $G\in \Gr\!/k$ be fixed and set $L:=\Lie(G,I)$.
Let $E_1,E_2\in\cE xt(G,L)$ be two extensions.
Their addition is given by Baer sum; here again this is
a two-step process. Namely, we
first form the fibre product $E'=E_1\times_G E_2$ which
is an extension of $G$ by $L\times L$. Then the Baer sum
is the pushforward of this extension along the addition
map $+:L\times L\to L$. All in all, we have the following
diagram which serves as a definition of $E_1+E_2$:
\[
\xymatrix{
1 \ar[r] & L\times L \ar[r] \ar[d]_-{+}
\ar@{}[rd]|{\mbox{\LARGE$\ulcorner$}\ }
& E' \ar[r] \ar[d] & G \ar[r] \ar@{=}[d] & 1 \\
1 \ar[r] & L \ar[r] & E_1+E_2 \ar[r] & G \ar[r] & 1.
}
\]
Explicitly, the underlying group scheme of the Baer
sum is given by $E_1+E_2=E'/M$ where
$$M=\ker(L\times L\to L)=\{(x,-x), x\in L\}.$$ The neutral
element for this addition is the trivial extension
$E=L\rtimes G$. Even though
this is not emphasized in the literature, the usual proofs
of the fact that the set of extensions endowed with the
Baer sum operation is an abelian group provide explicit
associativity and commutativity constraints proving that
$\cE xt(G,L)$ is a Picard category. The associativity
constraint is obtained by expressing
$(E_1+E_2)+E_3$ and $E_1+(E_2+E_3)$ as isomorphic quotients
of $E_1\times_G E_2\times_G E_3$, and the commutativity
constraint is obtained from the flipping morphism in
$E_1\times_G E_2$. The scalar multiplication
by $\lambda\in k$ is given by pushforward along the
multiplication-by-$\lambda$ morphism in the module scheme
$L=\Lie(G,I)$. All in all, we have the following diagram
which serves as a definition of $\lambda E$:
\[
\xymatrix{
1 \ar[r] & L \ar[r] \ar[d]_-{\lambda}
\ar@{}[rd]|{\mbox{\LARGE$\ulcorner$}}
& E \ar[r] \ar[d] & G \ar[r] \ar@{=}[d] & 1 \\
1 \ar[r] & L \ar[r] & \lambda E \ar[r] & G \ar[r] & 1.
}
\]
Again, the verification of the axioms of an $\OO_k$-module
stack is tedious but not difficult.
\end{item-title}

\subsection{Weil restriction generalities}
\label{subsection:Weil restriction}

We briefly give the main definitions and notations related to
Weil restriction; we refer to \cite[\S~7.6]{BLR90} for
more details. Let $h:\Spec(k')\to\Spec(k)$ be a finite,
locally free morphism of affine schemes. Let $(\Sch/k)$ be the
category of $k$-schemes and $(\Fun/k)$ the category of
functors $(\Sch/k)^\circ\to(\Sets)$. The Yoneda functor
embeds the former category into the latter. By sending a
morphism of functors $f:X'\to\Spec(k')$ to the morphism
$h\circ f:X'\to \Spec(k)$ we obtain a functor
$h_!:(\Fun/k')\to (\Fun/k)$. Sometimes we will refer to $h_!X'$
as {\em the $k'$-functor $X'$ viewed as a $k$-functor} and the
notation $h_!$ will be omitted. The pullback functor
$h^*:(\Fun/k)\to (\Fun/k')$ is right adjoint to $h_!$, in particular
$(h^*X)(S')=X(h_!S')$ for all $k'$-schemes~$S'$. The Weil
restriction functor $h_*:(\Fun/k')\to (\Fun/k)$ is right adjoint
to $h^*$, in particular we have $(h_*X')(S)=X'(h^*S)$ for all $k$-schemes
$S$. Thus we have a sequence of adjoint functors:
\[
h_!,h^*,h_*.
\]
The functors $h_!$ and $h^*$ preserve the subcategories of schemes.
The same is true for $h_*$ if $h$ is radicial, a case which covers
our needs (see \cite[\S~7.6]{BLR90} for refined representability
results). We write
\[
\alpha:1\too h_*h^*
\quad\mbox{and}\quad
\beta:h^*h_*\too 1
\]
for the unit and counit of the $(h^*,h_*)$-adjunction.
If $X$ is a separated $k$-scheme then $\alpha_X:X\to h_*h^*X$ is a
closed immersion. If $X'$ is a $k'$-group (resp. algebra)
functor (resp. scheme), then also $h_*X'$ is a $k$-group
(resp. algebra) functor (resp. scheme). If moreover
$X'\to \Spec(k')$ is smooth of relative dimension~$n$, then
$h_*X'\to \Spec(k)$ is smooth of relative dimension $n[k':k]$
where $[k':k]$ is the locally constant rank of $h$.
Quite often, it is simpler to consider functors defined on the
subcategory of affine schemes; the functors $h_!$, $h^*$, $h_*$
are defined similarly in this context.

\section{Group algebras of group schemes}
\label{group algebra}

Let $G$ be an affine $k$-group scheme. In this subsection, we
explain the construction of the group algebra $\OO_k[G]$,
which is the analogue in the setting of algebraic geometry
of the usual group algebra of abstract discrete groups.
Note that for a finite constant group scheme, the set $\OO_k[G](k)$
of $k$-rational points of the group algebra and the usual
group algebra $k[G]$ are isomorphic, but for other groups they
do not have much in common in general; this will be emphasized
below.
Since we are not aware of any appearance of the group algebra
in the literature, we give a somewhat detailed treatment,
including examples
(\ref{example:group algebra of finite groups}--\ref{examples:multiplicative group}),
basic properties (\ref{prop:group algebra}) and
the universal property
(\ref{theo:univ property gp algebra}).

\subsection{Linear algebra schemes}
\label{subsection:module points}

\begin{item-title}{Vector bundles}
\label{vector bundle envelope}
Let $S=\Spec(k)$.
As in \cite{EGA1-new}, we call {\em vector bundle}
an $\OO_k$-module scheme of the form $\VV(\mathscr{F})=\Spec \Sym(\mathscr{F})$
where $\mathscr{F}$ is a quasi-coherent $\cO_S$-module and $\Sym(\mathscr{F})$
is its symmetric algebra. We say {\em smooth vector bundle}
if $\mathscr{F}$ is locally free of finite rank.
If $f:X\to S$ is quasi-compact and quasi-separated, we set
$\VV(X/S):=\VV(f_*\cO_X)$ so
$\VV(X/S)(T)=\Hom_{\cO_T\mbox{-}\Mod}((f_*\cO_X)_T,\cO_T)$
for all $S$-schemes $T$. There is a canonical $S$-morphism
$\nu_X:X\too\VV(X/S)$
which is initial among all $S$-morphisms from $X$
to a vector bundle. We call it the {\em vector bundle envelope}
of $X/S$. If $X$ is affine over~$S$, the map $\nu_X$ is a closed
immersion because it is induced by the surjective morphism of
algebras $\Sym(f_*\cO_X)\too f_*\cO_X$ induced by the identity
$f_*\cO_X\too f_*\cO_X$.
\end{item-title}

\begin{item-title}{Base change and ring scheme}
\label{defs:base ring scheme}
If $h:\Spec(k')\to\Spec(k)$ is a morphism
of affine schemes, there is a morphism of $k'$-ring schemes
$h^*\OO_k\to\OO_{k'}$ which is the identity on points, hence
an isomorphism of ring schemes. However, whereas the target
has a natural structure of $\OO_{k'}$-algebra, the source
does not. For this reason, the pullback of module functors
or module schemes along $h$ is defined as
$M\mapsto h^*M\otimes_{h^*\OO_k}\OO_{k'}$, as
is familiar for the pullback of modules on ringed spaces.
Usually we will write simply $h^*M$.

For later use, we give some complements on the case where
$k'=k[I]$ is the ring of generalized dual numbers, for some
finite free $k$-module $I$.
We will identify $I$ and its dual $I^\smallvee=\Hom_k(I,k)$
with the coherent $\cO_{\Spec(k)}$-modules they define,
thus we have the vector bundle $\VV(I^\smallvee)$.
For each $k$-algebra $R$, we have $R[I]=R\oplus I\otimes_k R$
where $IR\simeq I\otimes_k R$ has square~0.
This decomposition functorial in $R$ gives rise to a direct
sum decomposition of $\OO_{k}$-module schemes:
\[
h_*\OO_{k[I]}=\OO_k\oplus \VV(I^\smallvee).
\]
It is natural to use the notation $\OO_k[\VV(I^\smallvee)]$
for the $\OO_k$-algebra scheme on the right-hand side,
however we will write more
simply $\OO_k[I]$. Now we move up on $\Spec(k[I])$, where
there is a morphism of $\OO_{k[I]}$-module
schemes $h^*\VV(I^\smallvee)\to\OO_{k[I]}$ defined for all
$k[I]$-algebras $R$ as the morphism $I\otimes_k R\to R$,
$i\otimes x\mapsto i_Rx$ where $i_R:=i1_R$. According to
what we said before, the pullback $h^*\VV(I^\smallvee)$ has the
module structure such that a section $a\in\OO_{k[I]}$ acts by
$a\cdot i\otimes x:=i\otimes ax$. The image of
$h^*\VV(I^\smallvee)\to\OO_{k[I]}$ is the ideal
$I\!\cdot\!\OO_{k[I]}$ defined by $(I\!\cdot\!\OO_{k[I]})(R)=IR$
for all $k[I]$-algebras $R$. There is an exact sequence
of $\OO_{k[I]}$-module schemes:
\[
0\too I\!\cdot\!\OO_{k[I]}\too \OO_{k[I]}\too i_*\OO_k\too 0.
\]
\end{item-title}

\begin{item-title}{$\OO_k$-Algebra schemes}
\label{defs:algebra schemes}
Here we define a category $(\OO_k\mbox{-}\Alg)$ of
{\em linear $\OO_k$-algebra schemes} and give a summary
of elementary properties.
For us an {\em $\OO_k$-algebra scheme} is a $k$-scheme $D$
endowed with two internal composition laws
$+,\times:D\times D\to D$ called addition and multiplication
possessing two neutral sections $0,1:\Spec(k)\to D$, and an external
law $\cdot:\OO_k\times D\to D$, such that for each $k$-algebra $R$
the tuple $(D(R),+,\times,0,1,\cdot)$ is an associative unitary
$R$-algebra. In particular $(D,+,0)$ is a commutative group
scheme, and $(D(R),\times,1)$ is a (possibly noncommutative)
monoid. We say that $D$ is a {\em linear $\OO_k$-algebra scheme}
if its underlying $\OO_k$-module scheme is a vector bundle.
In this case $\mathscr{F}$ can be recovered as the ``dual bundle''
sheaf $\mathscr{F}=\Hom_{\OO_k\mbox{-}\Mod}(D,\OO_k)$, the
Zariski sheaf over $S$ whose sections over an open~$U$ are the
morphisms of ${\OO_k}_{|U}$-modules $D_{|U}\to{\OO_k}_{|U}$
(read~\cite[9.4.9]{EGA1-new} between the lines). For an affine
$\OO_k$-algebra scheme in our sense, the comultiplication
is a map $\Delta:\Sym(\mathscr{F})\to \Sym(\mathscr{F})\otimes \Sym(\mathscr{F})$
and the bilinearity of $m$ implies that this map is
induced from a map $\Delta_0:\mathscr{F}\to\mathscr{F}\otimes\mathscr{F}$. Finally
we point out two constructions on linear $\OO_k$-algebra schemes.
The first is the tensor product $D\otimes_{\OO_k}D'$, which
as a functor is defined as $R\mapsto D(R)\otimes_RD'(R)$.
If $D=\VV(\mathscr{F})$ and $D'=\VV(\mathscr{F}')$ with $\mathscr{F},\mathscr{F}'$ locally free
of finite ranks, then
$D\otimes_{\OO_k}D'\simeq\VV(\mathscr{F}\otimes\mathscr{F}')$.
The second construction is the group of units. We
 observe that for any linear $\OO_k$-algebra scheme~$D$,
the subfunctor $D^\times\subset D$ of (multiplicative) units
is the preimage under the multiplication $D\times D\to D$
of the unit section $1:\Spec(k)\to D$ and is therefore
representable by an affine scheme. This gives rise to the
{\em group of units} functor
$(\OO_k\mbox{-}\Alg)\to (k\mbox{-}\Gr)$, $D\mapsto D^\times$
where $(k\mbox{-}\Gr)$ is the category of affine $k$-group schemes.

\begin{remark} \label{remark:def is restrictive}
We do not know if an
$\OO_k$-algebra scheme whose underlying scheme is affine
over $S$ is always of the form $\VV(\mathscr{F})$.
\end{remark}
\end{item-title}

\subsection{Group algebra: construction and examples}
\label{construction of gp alg}

Let $G=\Spec(A)$ be an affine $k$-group scheme. We write
$(u,v)\mapsto u\star v$ or sometimes simply $(u,v)\mapsto uv$
the multiplication of $G$. This operation extends to the
vector bundle envelope $\VV(G/k)$, as follows.
Let $\Delta:A\to A\otimes_k A$ be the comultiplication.
For each $k$-algebra $R$, we have
$\VV(G/k)(R)=\Hom_{k\mbox{-}\Mod}(A,R)$.
If $u,v:A\to R$ are morphisms of $k$-modules, we consider the
composition $u\star v:=(u\otimes v)\circ\Delta$:
\[
u\star v:A \stackrel{\Delta}{\tooo} A\otimes_kA
\stackrel{u\otimes v}{\tooo} R.
\]
Here the map $u\otimes v:A\otimes_k A\to R$ is
$a\otimes b\mapsto u(a)v(b)$.

\begin{definition}
The {\em group algebra of the $k$-group scheme $G$}:
\[
\OO_k[G]:=(\VV(G/k),+,\star),
\]
is the vector bundle $\VV(G/k)$ endowed with the product just
defined. We write $\nu_G:G\into \OO_k[G]$ for the closed
immersion as in paragraph~\ref{vector bundle envelope}.
\end{definition}

\bigskip

We check below that $\OO_k[G]$ is a linear $\OO_k$-algebra scheme.
Apart from $G(R)$, there is another noteworthy subset inside
$\OO_k[G](R)$, namely the set $\Der_G(R)$ of $k$-module maps $d:A\to R$
which are $u$-derivations for some $k$-algebra map $u:A\to R$
(which need not be determined by~$d$); a more accurate notation
would be $\Der(\cO_G,\OO_k)(R)$ but we favour lightness.
Here are the first basic properties coming out of the construction.

\begin{prop} \label{prop:elem properties of group algebra}
Let $G$ be an affine $k$-group scheme. Let $\OO_k[G]$
and $\Der_G$ be as described above.
\begin{trivlist}
\itemn{1} The tuple $\OO_k[G]:=(\VV(G/k),+,\star)$ is
a linear $\OO_k$-algebra scheme.
\itemn{2} As a $k$-scheme, $\OO_k[G]$ is $k$-flat (resp. has
$k$-projective function ring) iff $G$ has the same property.
\itemn{3} The composition $G\into \OO_k[G]^\times\into\OO_k[G]$
is a closed embedding, hence also $G\into \OO_k[G]^\times$.
\itemn{4} The subfunctor $\Der_G\subset \OO_k[G]$ is stable
by multiplication by $G$ on the left and on the right, so
it acquires left and right $G$-actions.
More precisely, if $u,v:A\to R$ are maps of algebras,
$d:A\to R$ is a $u$-derivation and $d':A\to R$ is a
$v$-derivation, then $d\star v$ and $u\star d'$ are
$(u\star v)$-derivations.
\end{trivlist}
\end{prop}

\begin{proof}
Omitted.
\hfill $\Box$
\end{proof}

\begin{item-title}{Remark on Hopf algebra structure}
If $G$ is a finite, locally free commutative $k$-group scheme,
then the multiplication of its function ring induces a
comultiplication on $\OO_k[G]$ making it into a
Hopf $\OO_k$-algebra scheme. Moreover the $k$-algebra
$\OO_k[G](k)$ is the ring of functions of the Cartier dual
$G^\smallvee$. This Hopf algebra structure highlights
Examples~\ref{examples:additive group}
and~\ref{examples:multiplicative group} below.
\end{item-title}

\begin{example-title}{1}{finite locally free groups}
\label{example:group algebra of finite groups}
If $G$ is a finite locally free $k$-group scheme, the algebra
$\OO_k[G]$ is a smooth vector bundle of rank~$[G:\Spec(k)]$
and its group of units is the complement
in $\OO_k[G]$ of the Cartier divisor equal to the zero locus
of the determinant of the left regular representation:
\[
\OO_k[G]\stackrel{\LL}{\intoo}\End_{\OO_k\mbox{-}\Mod}(\OO_k[G])
\stackrel{\det}{\tooo} \AA^1.
\]
If moreover $G$ is the finite constant $k$-group scheme
defined by a finite abstract group $\Gamma$, then $\OO_k[G]$
is isomorphic to the algebra scheme defined by the abstract
group algebra $k[\Gamma]$, that is to say
$\OO_k[G](R)\simeq R[\Gamma]$ functorially in $R$.
\end{example-title}

\begin{example-title}{2}{the additive group}
\label{examples:additive group}
Let $G=\GG_a=\Spec(k[x])$. For a $k$-algebra $R$, let
\[
R\langle\!\langle t\rangle\!\rangle
=R[[t,t^{[2]},t^{[3]},\dots]]
\]
be the $R$-algebra of divided power formal power series in one
variable $t$, with $t^{[i]}t^{[j]}={i+j\choose i}t^{[i+j]}$
for all $i,j\ge 0$. Setting
$\OO\langle\!\langle t\rangle\!\rangle(R)=
R\langle\!\langle t\rangle\!\rangle$,
we have an isomorphism
$\OO_k[G]\isomto \OO\langle\!\langle t\rangle\!\rangle$ given
by:
\[
\OO_k[G](R) \isomto R\langle\!\langle t\rangle\!\rangle
\quad,\quad
(k[x]\stackrel{u}{\too} R) \longmapsto \sum_{i\ge 0} u(x^i)t^i.
\]
If $k$ is a ring of characteristic $p>0$, let $H=\alpha_p$
be the kernel of Frobenius. The algebra
$\OO_k[H](R)$ is identified with the $R$-subalgebra of
$R\langle\!\langle t\rangle\!\rangle$ generated by $t$, which
is isomorphic to $R[t]/(t^p)$ because $t^p=pt^{[p]}=0$.
\end{example-title}

\begin{example-title}{3}{the multiplicative group}
\label{examples:multiplicative group}
Let $G=\GG_m=\Spec(k[x,1/x])$. Let $R^{\ZZ}$ be
the product algebra, whose elements are sequences with componentwise
addition and multiplication. We have an isomorphism
$\OO_k[G]\isomto \prod_{i\in\ZZ}\OO_k$ given by:
\[
\OO_k[G](R) \isomto R^{\ZZ}
\quad,\quad
(k[x]\stackrel{u}{\too} R) \longmapsto \{u(x^i)\}_{i\in\ZZ}.
\]
More generally, for any torus $T$ with
character group $X(T)$, we have
$\OO_k[T]\isomto \prod_{i\in X(T)}\OO_k$.
Let $H=\mu_n$ be the subgroup of $n$-th roots of
unity. The algebra $\OO_k[H](R)$ is identified with
the $R$-subalgebra of $\ZZ/n\ZZ$-invariants
$\OO_k[G](R)^{\ZZ/n\ZZ}$, composed of sequences $\{r_i\}_{i\in\ZZ}$
such that $r_{i+nj}=r_i$ for all $i,j\in\ZZ$.
\end{example-title}

\subsection{Properties: functoriality and adjointness}

Here are some functoriality properties of the
group algebra.

\begin{prop} \label{prop:group algebra}
Let $G$ be an affine $k$-group scheme.
The formation of the group algebra $\OO_k[G]$:
\begin{trivlist}
\itemn{1} is functorial in $G$ and faithful;
\itemn{2} commutes with base change $k'/k$;
\itemn{3} is compatible with products: there is a canonical
isomorphism
$\OO_k[G]\otimes_{\OO_k}\OO_k[H]\isomto\OO_k[G\times H]$;
\itemn{4} is compatible with Weil restriction: if
$h:\Spec(k')\to \Spec(k)$ is a finite locally free morphism
of schemes, there is an isomorphism of $\OO_k$-algebra schemes
$\OO_k[G]\otimes_{\OO_k}h_*\OO_{k'}\isomto h_*h^*\OO_k[G]$.
\end{trivlist}
\end{prop}

\begin{proof}
(1) Any map of affine $k$-group schemes
$G=\Spec(A)\to H=\Spec(B)$
gives rise to a map of $k$-Hopf algebras $B\to A$
and then by precomposition to a map of $R$-algebras
$\OO_k[G](R)\to \OO_k[H](R)$ which is functorial in $R$.
Faithfulness follows from the fact that $G\into\OO_k[G]$
is a closed immersion.

\smallskip

\noindent (2) The isomorphism
$\Hom_{k'\mbox{-}\Mod}(A\otimes_kk',R')\isomto
\Hom_{k\mbox{-}\Mod}(A,R')$, functorial in the $k'$-algebra $R'$,
gives an isomorphism
$\OO_{k'}[G_{k'}]\isomto \OO_k[G]\times_{\Spec(k)}\Spec(k')$.

\smallskip

\noindent (3) Write $G=\Spec(A)$ and $H=\Spec(B)$.
To a pair of $k$-module maps $u:A\to R$ and $v:B\to R$ we
attach the map $u\otimes v:A\otimes_kB\to R$,
$a\otimes b\mapsto u(a)v(b)$. This defines an isomorphism
\[ 
\Hom_{k\mbox{-}\Mod}(A,R)\otimes_R\Hom_{k\mbox{-}\Mod}(B,R)
\isomto \Hom_{k\mbox{-}\Mod}(A\otimes_kB,R)
\]
which is functorial in $R$. The result follows.

\smallskip

\noindent (4) If $D$ is an $\OO_k$-algebra scheme, the 
functorial $R$-algebra maps
\[
\setlength{\arraycolsep}{.6mm}
\begin{array}{rcl}
D(R)\otimes_R (R\otimes_kk') & \too & D(R\otimes_kk') \\
d\otimes r' & \longmapsto & r'd
\end{array}
\]
fit together to give a morphism
$D\otimes_{\OO_k}h_*\OO_{k'}\to h_*h^*D$.
In case $D=\OO_k[G]$ and $h$ finite locally free, this is
none other than the  isomorphism
$\Hom_{k\mbox{-}\Mod}(A,R)\otimes_k k'\isomto
\Hom_{k\mbox{-}\Mod}(A,R\otimes_kk')$.
\hfill $\Box$
\end{proof}

Finally we prove the adjointness property. We recall
from paragraph~\ref{defs:algebra schemes}
(see also Remark~\ref{remark:def is restrictive})
that $(\OO_k\mbox{-}\Alg)$
is the category of $\OO_k$-algebra schemes whose
underlying $\OO_k$-module scheme is of the form
$\VV(\mathscr{F})=\Spec \Sym(\mathscr{F})$ for some quasi-coherent
$\cO_{\Spec(k)}$-module $\mathscr{F}$, and that $(k\mbox{-}\Gr)$
is the category of affine $k$-group schemes.

\begin{theocite}{Adjointness property of the group algebra}
\label{theo:univ property gp algebra}
The group algebra functor is left adjoint to the group of
units functor. In other words, for all affine $k$-group schemes
$G$ and all linear $\OO_k$-algebra schemes $D$, the map that
sends a morphism of algebra schemes $\OO_k[G]\to D$ to the
composition $G\subset \OO_k[G]^\times\to D^\times$ gives a
bifunctorial bijection:
\[
\Hom_{\OO_k\mbox{-}\Alg}(\OO_k[G],D)\isomto
\Hom_{k\mbox{-}\Gr}(G,D^\times).
\]
\end{theocite}

\begin{proof}
We describe a map in the other direction and we leave to the
reader the proof that it is an inverse. Let $f:G\to D^\times$
be a morphism of $k$-group schemes. We will construct a map of
functors $f':\OO_k[G]\to D$. We know from
paragraph~\ref{defs:algebra schemes} that $D=\VV(F)$ where
$F$ is a $k$-module, and that the comultiplication
$\Delta_D:\Sym(F)\to \Sym(F)\otimes \Sym(F)$ is induced by a
morphism $\Delta_0:F\to F\otimes F$. Consider the composition
$G\to D^\times \subset D$ and let $g:\Sym(F)\to A$ be the
corresponding map of algebras. For each $k$-algebra $R$ we
have the equalities $\OO_k[G](R)=\Hom_{k\mbox{-}\Mod}(A,R)$ and
$D(R)=\Hom_{k\mbox{-}\Mod}(F,R)$. We define $f'$ as follows:
\[
\setlength{\arraycolsep}{.6mm}
\begin{array}{rcl}
\OO_k[G](R) & \too & D(R) \\
(u:A\to R) & \longmapsto & (u\circ g\circ i:F\to R)
\end{array}
\]
where $i:F\into\Sym(F)$ is the inclusion as the degree 1
piece in the symmetric algebra. The map $f'$ is a map of
modules, and we only have to check that it respects the
multiplication. Let $u,v:A\to R$ be module homomorphisms.
We have the following commutative diagram:
\[
\xymatrix@C=15mm{
A \ar[r]^-{\Delta_G} & A\otimes A \ar[r]^-{u\otimes v} & R \\
\Sym(F) \ar[r]^-{\Delta_D} \ar[u]^g & \Sym(F)\otimes \Sym(F)
\ar[u]_{g\otimes g} & \\
F \ar[r]^{\Delta_0} \ar[u]^i & F\otimes F \ar[u]_{i\otimes i} &}
\]
With the $\star$ notation as in
Subsection~\ref{construction of gp alg},
we compute:
\begin{align*}
(u\star_G v)\circ g\circ i
& =(u\otimes v)\circ \Delta_G\circ g\circ i \\
& =(ugi\otimes vgi)\circ \Delta_0 \\
& =ugi\star_D vgi. \\
\end{align*}
Thus $f':\OO_k[G]\to D$ is a map of algebra schemes and
this ends the construction.
\hfill $\Box$
\end{proof}

\begin{remark}
It follows from this result that a smooth vector bundle
with action of $G$ is the same as a (smooth) $\OO_k[G]$-module
scheme. Indeed, the endomorphism algebra of such a vector
bundle is representable by a linear $\OO_k$-algebra scheme $D$.
For example, if $G$ is an affine, finite type,
differentially flat $k$-group scheme, then the adjoint
action on $\Lie G$ makes it an $\OO_k[G]$-module scheme.
\end{remark}

\section{Weil restriction}
\label{section:Weil restriction}

We keep the notations from the previous sections.
In this section, we describe the Weil restriction $E=h_*\mathscr{G}$
of a $k[I]$-group scheme $\mathscr{G}\in\Gr\!/k[I]$ and show how it
carries the structure of an object of the category of
extensions $\Ext(I)/k$. The necessary notions
of differential calculus (tangent bundle, Lie algebra, exponential)
are recalled in Appendix~\ref{section:differential calculus}.

\subsection{Weil restriction of the group algebra}

If $\mathscr{G}$ is an affine $k[I]$-group scheme, its Weil restriction
$h_*\mathscr{G}$ embeds in the pushforward algebra $h_*\OO_{k[I]}[\mathscr{G}]$.
(It also embeds in the algebra $\OO_k[h_*\mathscr{G}]$ which however
is less interesting in that it does not reflect the Weil
restriction structure.) Our aim in this subsection is to give
a description of $h_*\OO_{k[I]}[\mathscr{G}]$ suited to the computation of
the adjunction map $\beta_{\mathscr{G}}$.

The starting point is the following definition and lemma.
Let $A$ be a $k[I]$-algebra and $R$ a $k$-algebra.
Let $v:A\to I\otimes_k R$ be a $k$-linear map such that
there is a $k$-linear map ${\bar v}:A\to R$ satisfying
$v(ix)=i\bar v(x)$ for all $i\in I$ and $x\in A$. Then ${\bar v}$
is uniquely determined by $v$; in fact it is already
determined by the identity $v(ix)=i{\bar v}(x)$ for any
fixed~$i$ belonging to a basis of $I$ as a $k$-module.

\begin{definition} \label{definition:$I$-compatible maps}
Let $A$ be a $k[I]$-algebra and $R$ a $k$-algebra. We say
that a $k$-linear map $v:A\to I\otimes_k R$ is {\em $I$-compatible}
if there is a $k$-linear map ${\bar v}:A\to R$ such that
$v(ix)=i{\bar v}(x)$ for all $i\in I$ and $x\in A$. We denote by
$\Homc_k(A,I\otimes_k R)$ the $R$-module of $I$-compatible maps
and $\Homc_k(A,I\otimes_k R)\to\Hom_k(A,R)$, $v\mapsto {\bar v}$
the $R$-module map that sends $v$ to the unique map ${\bar v}$
with the properties above.
\end{definition}

\bigskip

Note that since $I^2=0$ in $R[I]$, if $v$ is $I$-compatible then
${\bar v}$ vanishes on $IA$. In other words, the map $v\mapsto {\bar v}$
factors through $\Hom_k(A/IA,R)$.

\begin{lemma} \label{lemma:description of Weil res isom}
Let $A$ be a $k[I]$-algebra and $R$ a $k$-algebra.
\begin{trivlist}
\itemn{1} Each morphism of $k[I]$-modules $f:A\to R[I]$
is of the form $f={\bar v}+v$ for a unique $I$-compatible
$k$-linear map $v:A\to I\otimes_k R$, and conversely.
\itemn{2} Each morphism of $k[I]$-algebras $f:A\to R[I]$
is of the form $f={\bar v}+v$ as above with $v$ satisfying
moreover $v(i)=i$ for all $i\in I$ and $v(xy)={\bar v}(x)v(y)+{\bar v}(y)v(x)$
for all $x,y\in A$, and conversely. In particular
${\bar v}:A\to R$ is a $k$-algebra homomorphism and
$v:A\to I\otimes_k R$ is a ${\bar v}$-derivation.
\end{trivlist}
\end{lemma}

\begin{proof}
(1) Using the decomposition $R[I]=R\oplus I\otimes_k R$,
we can write $f(x)=u(x)+v(x)$ for some unique $k$-linear
maps $u:A\to R$ and $v:A\to I\otimes_k R$. Then $f$ is
$k[I]$-linear if and only if $f(ix)=if(x)$ for all $i\in I$
and $x\in A$. Taking into account that $I^2=0$, this means
that $v$ is $I$-compatible and $u={\bar v}$.

\smallskip

\noindent (2) The condition $f(1)=1$ means that ${\bar v}(1)=1$,
that is $v(i)=i$ for all $i\in I$, and $v(1)=0$.
The condition of multiplicativity of~$f$ means that
${\bar v}$ is multiplicative and $v$ is a ${\bar v}$-derivation, i.e.
$v(xy)={\bar v}(x)v(y)+{\bar v}(y)v(x)$. In the presence of the
derivation property, the multiplicativity of ${\bar v}$ is
automatic (computing $v(ixy)$ in two different ways) as well
as the condition $v(1)=0$ (setting $x=y=1$). Conversely if $v$
is $I$-compatible with $v(i)=i$ and $v(xy)={\bar v}(x)v(y)+{\bar v}(y)v(x)$,
one sees that ${\bar v}$ is a morphism
of rings and $f={\bar v}+v$ is a morphism of $k[I]$-algebras.
\hfill $\Box$
\end{proof}

Now let $\mathscr{G}$ be an affine $k[I]$-group scheme.
Lemma~\ref{lemma:description of Weil res isom} shows that
the Weil restriction $h_*\VV(\mathscr{G}/k[I])$ can be described
in terms of the scheme of $I$-compatible maps, defined as a functor
on $k$-algebras by:
\[
\OO_{\mathrm{c}}(\mathscr{G})(R):=\Homc_k(A,I\otimes_k R).
\]
We know that $h_*\VV(\mathscr{G}/k[I])$ supports the algebra scheme
structure $h_*\OO_{k[I]}[\mathscr{G}]$, and we will now identify
the multiplication induced on
$\OO_{\mathrm{c}}(\mathscr{G})$ by means of this isomorphism.

\begin{prop} \label{prop:group algebra O(h!)}
Let $\mathscr{G}=\Spec(A)$ be an affine $k[I]$-group scheme
with comultiplication $\Delta:A\to A\otimes_{k[I]}A$
and counit $e:A\to k[I]$, with $e={\bar d}+d$
for a unique $I$-compatible $k$-linear map $d:A\to I$.
\begin{trivlist}
\itemn{1}
Let $R$ be a $k$-algebra and let $v,w:A\to I\otimes_k R$
be two $I$-compatible $k$-linear maps. Then the morphism
${\bar v}\otimes_k w+v\otimes_k {\bar w}:A\otimes_kA\to R$
factors through a well-defined $k$-linear morphism
\[
{\bar v}\otimes_k w+v\otimes_k {\bar w}:A\otimes_{k[I]}A\to R.
\]
\itemn{2} For $v,w$ as before let:
\[
v\diamond w:=({\bar v}\otimes_k w+v\otimes_k {\bar w})\circ\Delta.
\]
Then $(\OO_{\mathrm{c}}(\mathscr{G}),+,\diamond)$ is an associative unitary
$\OO_k$-algebra with multiplicative unit $d$, and the map
\[
\theta_\mathscr{G}:\OO_{\mathrm{c}}(\mathscr{G})\isomto h_*\OO_{k[I]}[\mathscr{G}],\quad
v\longmapsto {\bar v}+v
\]
is an isomorphism of associative unitary $\OO_k$-algebras.
\end{trivlist}
\end{prop}

\begin{proof}
(1) The $k$-linear morphism ${\bar v}\otimes_k w+v\otimes_k {\bar w}$
takes the same value $i{\bar v}(a){\bar w}(b)$ on the tensors
$ia\otimes b$ and $a\otimes i b$ for all $i\in I$, $a,b\in A$.
Therefore it vanishes on
tensors of the form $(a\otimes b)(i\otimes 1-1\otimes i)$.
Since these tensors generate the kernel of the ring map
$A\otimes_kA\to A\otimes_{k[I]}A$, we obtain an induced morphism
${\bar v}\otimes_k w+v\otimes_k {\bar w}:A\otimes_{k[I]}A\to R$.

\smallskip

\noindent (2) According to (1) the definition of $v\diamond w$
makes sense. For the rest of the statement, it is enough to prove that
$\theta_\mathscr{G}(v\diamond w)=\theta_\mathscr{G}(v)\star \theta_\mathscr{G}(w)$
because if this is the case then all the known properties
of the product $\star$ in $h_*\OO_{k[I]}[\mathscr{G}]$ are transferred to
$\diamond$ by the isomorphism~$\theta_\mathscr{G}$. On one hand,
using the expression
$v\diamond w=({\bar v}\otimes_k w+v\otimes_k {\bar w})\circ\Delta$
and the fact that $\Delta$ is $k[I]$-linear, we find
$\overline{v\diamond w}=({\bar v}\otimes_k{\bar w})\circ\Delta$ hence:
\begin{align*}
\theta_\mathscr{G}(v\diamond w)
& =
\big({\bar v}\otimes_k {\bar w}\big)\circ\Delta+
\big({\bar v}\otimes_k w+v\otimes_k {\bar w}\big)\circ\Delta \\
& = \big[{\bar v}\otimes_k {\bar w}+
({\bar v}\otimes_k w+v\otimes_k {\bar w})\big]\circ\Delta.
\end{align*}
On the other hand, we have:
\[
\theta_\mathscr{G}(v)\star \theta_\mathscr{G}(w)=
\big[({\bar v}+v)\otimes_k ({\bar w}+w)\big]\circ\Delta.
\]
The maps in the brackets are equal, whence
$\theta_\mathscr{G}(v\diamond w)=\theta_\mathscr{G}(v)\star \theta_\mathscr{G}(w)$
as desired.
\hfill $\Box$
\end{proof}

\begin{remark}
If $(\eps_1,\dots,\eps_r)$ is a basis of $I$ we have
a concrete description as follows. A $k$-linear map
$v:A\to I\otimes_k R$ can be written
$v=\eps_1v_1+\dots+\eps_rv_r$ for some maps
$v_j:A\to R$. Then $v$ is $I$-compatible if and only if
$v_j\eps_i=\delta_{i,j}{\bar v}$ for all $i,j$.
If this is the case, $v_j$ induces a $k$-linear morphism
$A\otimes_{k[I]}k[\eps_j]\to R$ and ${\bar v}=v_j\eps_j$ for each $j$.
Now write
$\mathscr{G}_j=\mathscr{G}\otimes_{k[I]}k[\eps_j]$ and
$h_j:\Spec(k[\eps_j])\to\Spec(k)$ the structure map.
Also let $\mathscr{G}_k=\mathscr{G}\otimes_{k[I]}k$ so we have maps
$\VV(h_{j,!}\mathscr{G}_j/k)\to \VV(\mathscr{G}_k/k)$, $v_j\mapsto {\bar v}:=v_j\eps_j$.
Then we have an isomorphism:
\[
\VV(h_{1,!}\mathscr{G}_1/k)\underset{\VV(\mathscr{G}_k/k)}{\times}\dots
\underset{\VV(\mathscr{G}_k/k)}{\times} \VV(h_{r,!}\mathscr{G}_r/k)
\isomto \OO_{\mathrm{c}}(\mathscr{G})
\]
given by
$(v_1,\dots,v_r) \mapsto v=\eps_1v_1+\dots+\eps_rv_r$.
\end{remark}

\subsection{Kernel of the adjunction $\beta:h^*h_*\mathscr{G}\to \mathscr{G}$}
\label{subsection:the map beta}

Again let $\mathscr{G}=\Spec(A)$ be an affine $k[I]$-group scheme.
Denote by $\Delta:A\to A\otimes_{k[I]}A$ the comultiplication
and $e:A\to k[I]$ the counit, of the form $e={\bar d}+d$
for a unique $I$-compatible $k$-linear map $d:A\to I$.
The purpose of this subsection is to prove the
following proposition.

\begin{prop} \label{ker of beta}
Let $\mathscr{G}$ be an affine $k[I]$-group scheme and
$\mathscr{G}_k=i^*\mathscr{G}$. Let $\beta=\beta_\mathscr{G}:h^*h_*\mathscr{G} \to \mathscr{G}$ be the
adjunction, and $\mathrm{L}(\mathscr{G}):=\ker(\beta)$.
\begin{trivlist}
\itemn{1} {\rm Functoriality.} The formation of $\mathrm{L}(\mathscr{G})$
is functorial for morphisms of pointed $k[I]$-schemes, and for
morphisms of $k[I]$-group schemes.
\itemn{2} {\rm Explicit description.}
We have embeddings of monoids
$\mathrm{L}(\mathscr{G})\subset h^*h_*\mathscr{G} \subset (h^*\OO_{\mathrm{c}}(\mathscr{G}),\diamond)$
under which, functorially in the $k[I]$-algebra $R$:
\[
(h^*h_*\mathscr{G})(R)
=\left\{v \in (h^*\OO_{\mathrm{c}}(\mathscr{G}))(R) \
\left| \
\begin{array}{l}
v(i)=i \mbox{ for all } i\in I \\
v(xy)={\bar v}(x)v(y)+{\bar v}(y)v(x) \mbox{ for all } x,y\in A
\end{array}
\right.
\right\}
\]
and if $v_R:A\to R$ denotes the composition
$A\stackrel{v}{\too} I\otimes_k R\stackrel{i\otimes r\mapsto ir}{\tooo}R$:
\[
\mathrm{L}(\mathscr{G})(R)=
\big\{ v\in (h^*h_*\mathscr{G})(R) \ : \ {\bar v}+v_R=e_R \big\}.
\]
\itemn{3} {\rm Special fibre.}
There is a functorial isomorphism
$i^*\mathrm{L}(\mathscr{G})\isomto \Lie(\mathscr{G}_k,I)$,
$v\mapsto v-d$.
Under this isomorphism,
the action by conjugation of $h_*\mathscr{G}=i^*h^*h_*\mathscr{G}$ on
$i^*\mathrm{L}(\mathscr{G})$ is given by the morphism:
\[
h_*\mathscr{G} \stackrel{i^*\beta}{\tooo} \mathscr{G}_k \stackrel{\Ad}{\tooo}
\GL(\Lie(\mathscr{G}_k,I)).
\]
\itemn{4} {\rm Case of trivial deformation groups.}
If $\mathscr{G}=h^*G$ for some affine $k$-group scheme $G$,
there is a canonical and functorial isomorphism
$\mathrm{L}(\mathscr{G})\isomto h^*\Lie(G,I)$.
More precisely, let $\exp_{G,I}:h^*\Lie(G,I) \to h^*G$
be the exponential morphism as defined in~\ref{def:exponential}.
Then under the isomorphism (see~\ref{TG and Lie G})
\[
\varrho_G:\Lie(G,I)\times G \isomto \T(G,I),
\]
the subgroup $\mathrm{L}(\mathscr{G})\subset h^*\T(G,I)$ has for points
the pairs $(x,g) \in h^*\Lie(G,I)\times h^*G$ such that
$g=\exp(-x)$, and the isomorphism is given by $(x,g)\mapsto x$.
\end{trivlist}
\end{prop}

\begin{proof}
(1) If $\varphi:\mathscr{G}\to \mathscr{G}'$ is a morphism of
pointed schemes, then by functoriality of $\beta$ the morphism
$h_*h^*\varphi$ takes the kernel of $\beta_\mathscr{G}$ into the kernel of
$\beta_{\mathscr{G}'}$. If moreover $\varphi$ is a map of group schemes
then the restriction of $h_*h^*\varphi$ to $\mathrm{L}(\mathscr{G})$ also.

\smallskip

\noindent (2)
In the rest of the proof we use the possibility to compute
in the group algebra $(\OO_{\mathrm{c}}(\mathscr{G}),+,\diamond)$,
see Proposition~\ref{prop:group algebra O(h!)}.
 The description of $h^*h_*\mathscr{G}$ as a submonoid of the
multiplicative monoid of $\OO_{\mathrm{c}}(\mathscr{G})$ is copied from
Lemma~\ref{lemma:description of Weil res isom}. The description of
$\mathrm{L}(\mathscr{G})$ follows from the fact that for $f\in\mathscr{G}(R[I])$,
$f={\bar v}+v:A\to R[I]$, the image $\beta(f)\in\mathscr{G}(R)$,
$\beta(f):A\to R$ is the map ${\bar v}+v_R$.

\smallskip

\noindent (3) -- first claim.
The pullback $i^*$ is the restriction to the category of those
$k[I]$-algebras $R$ such that $IR=0$. For such an $R$,
an element $v\in (i^*\mathrm{L}(\mathscr{G}))(R)$ is a $k$-linear map
$v:A\to R$ such that $v(i)=i$, $v(xy)={\bar v}(x)v(y)+{\bar v}(y)v(x)$
and ${\bar v}=e_R$. In particular we see that $v$ is an
$e_R$-derivation. Since also $d_R$ is an $e_R$-derivation with
$d_R(i)=i$ for all $i\in I$, the difference $\delta:=v-d_R$
induces an $e_R$-derivation $\delta:A\to I\otimes_kR$ vanishing
on $IA$, i.e. an $R$-point of $\Lie(\mathscr{G}_k,I)$. Conversely, any
$e_R$-derivation $\delta:A\to I\otimes_kR$ vanishing on $IA$
gives rise to a $k$-linear map $v:A\to I\otimes_kR$
defined by $v:=d_R+\delta$ and satisfying the properties required
to be a point of $(i^*\mathrm{L}(\mathscr{G}))(R)$. Finally let
$\delta_1,\delta_2\in \Lie(\mathscr{G}_k,I)(R)$. Since $\delta_1,\delta_2$
vanish on $IA$, we have $d_R^*+\delta_1^*=d_R^*+\delta_2^*=e_R$
and then:
\begin{align*}
(d_R+\delta_1)\diamond (d_R+\delta_2)
& = \big[e_R\otimes (d_R+\delta_2)+(d_R+\delta_1)
\otimes e_R\big]\circ\Delta \\
& = \big[(e_R\otimes d_R+d_R\otimes e_R)+
(e_R\otimes \delta_2+\delta_1\otimes e_R)\big]\circ\Delta.
\end{align*}
All three morphisms $e_R\otimes d_R+d_R\otimes e_R$,
$e_R\otimes \delta_2$ and $\delta_1\otimes e_R$ factor through
$A\otimes_{k[I]}A$, so the precomposition with $\Delta$
is distributive for them. Since also $d_R$ is the neutral
element for the law $\diamond$ and $e_R$ is the neutral
element for the law $\star$, we obtain:
\begin{align*}
(d_R+\delta_1)\diamond (d_R+\delta_2)
& = \big[e_R\otimes d_R+d_R\otimes e_R\big]\circ\Delta
+\big[e_R\otimes \delta_2\big]\circ\Delta
+\big[\delta_1\otimes e_R\big]\circ\Delta \\
& = d_R+\delta_1+\delta_2.
\end{align*}
This shows that the isomorphism
$i^*\mathrm{L}(\mathscr{G})\isomto \Lie(\mathscr{G}_k)$,
$v\mapsto v-d_R$ is a morphism of groups.

\smallskip

\noindent (3) -- second claim. Let $R$ be a $k[I]$-algebra such that
$IR=0$. The action of an element $u\in (h^*h_*\mathscr{G})(R)$ on an
element $v=d_R+\delta\in \mathrm{L}(\mathscr{G})(R)$ by conjugation
can be computed in the algebra $h^*\OO_{\mathrm{c}}(\mathscr{G})$:
\[
u\diamond v\diamond u^{-1}=u\diamond (d_R+\delta)\diamond u^{-1}
=u\diamond d_R\diamond u^{-1}+u\diamond \delta\diamond u^{-1}
=d_R+u\diamond \delta\diamond u^{-1}.
\]
We see that $u$ is acting on $\delta$ by conjugation in the
group algebra. This is the adjoint action, as explained
in Proposition~\ref{prop:embeddings in OG}.

\smallskip

\noindent (4) When $\mathscr{G}=h^*G$, the adjunction map is the
infinitesimal translation as in
Proposition~\ref{prop:elementary properties of exp}(3).
The $R$-points of $\mathrm{L}(\mathscr{G})$ are the pairs
$(x,g)$ such that $\exp(x)g=1$ in $G(R)$.
This amounts to $g=\exp(-x)$ which proves (4).
\hfill $\Box$
\end{proof}

\subsection{Extension structure of the Weil restriction}
\label{subsection:extension structure}

Let $\mathscr{G}$ be a $k[I]$-group scheme. The notion of
{\em rigidification} for $\mathscr{G}$ and the property that
$\mathscr{G}$ be {\em rigid} are defined in \ref{item:results}.
Here are some remarks.

\begin{trivlist}
\itemn{1} If $\sigma:h^*\mathscr{G}_k\isomto \mathscr{G}$ is a rigidification,
then $\sigma(1)^{-1}\cdot \sigma$ is another. Therefore if there
exists a rigidification~$\sigma$, we may always assume moreover
that $\sigma(1)=1$.
\itemn{2} By the infinitesimal lifting criterion, all
smooth affine $k[I]$-group schemes are rigid. By Cartier
duality, $k[I]$-group schemes of multiplicative type are
rigid.
\itemn{3} If $\alpha:\mathscr{G}\to \mathscr{G}'$ is a morphism
between rigid $k[I]$-group schemes, it is not always possible
to choose rigidifications for $\mathscr{G}$ and $\mathscr{G}'$
that are compatible in the sense that
$\sigma'\circ h^*\alpha=\alpha\circ \sigma$.
For instance let $I=k\eps$ and let $\alpha:\GG_a\to \GG_a$ be
the morphism defined by $\alpha(x)=\eps x$. Then $\mathscr{G}$ and
$\mathscr{G}'$ are rigid but since $h^*\alpha=0$, there do not
exist compatible rigidifications.
\end{trivlist}

\begin{lemma} \label{lemma:rigidifications vs sections}
Let $\mathscr{G}$ be a $k[I]$-group scheme such that the
restriction homomorphism $\mathscr{G}_k=i^*\mathscr{G}$ is $k$-flat. Let
$\pi:=i^*\beta_\mathscr{G}:h_*\mathscr{G}\to\mathscr{G}_k$. Then the adjunction
$\Hom_{\Sch/k[I]}(h^*\mathscr{G}_k,\mathscr{G})=\Hom_{\Sch/k}(\mathscr{G}_k,h_*\mathscr{G})$
induces a bijection between rigidifications of $\mathscr{G}$ and sections
of $\pi$.
\end{lemma}

\begin{proof}
Let $\sigma:h^*\mathscr{G}_k\isomto \mathscr{G}$ be a rigidification and
$s:=h_*\sigma\circ\alpha_{\mathscr{G}_k}$.
Then we have $\sigma=\beta_\mathscr{G}\circ h^*s$ and by applying $i^*$
we find $\id_{\mathscr{G}_k}=i^*\sigma=\pi\circ s$ hence $s$ is a section
of $\pi$. Conversely let $s$ be a section of $\pi$ and
$\sigma:=\beta_\mathscr{G}\circ h^*s$. Then
$i^*\sigma=i^*\beta_\mathscr{G}\circ i^*h^*s=\pi \circ s=\id_{\mathscr{G}_k}$ hence
$\sigma$ lifts the identity. In particular $\sigma$ is an affine
morphism. Since moreover $h^*\mathscr{G}_k$ is flat, we conclude that
$\sigma$ is an isomorphism, hence a rigidification.
\hfill $\Box$
\end{proof}

\begin{lemma} \label{lemma:beta for rigid groups}
Let $\mathscr{G}$ be an affine, differentially flat and rigid
$k[I]$-group scheme. Then $\beta:h^*h_*\mathscr{G}\to \mathscr{G}$
is faithfully flat and we have an exact sequence:
\[
1 \too \mathrm{L}(\mathscr{G}) \too h^*h_*\mathscr{G} \stackrel{\beta}{\too} \mathscr{G} \too 1.
\]
If $\mathscr{G}$ is of finite type over $k[I]$, then $\beta$
is smooth.
\end{lemma}

\begin{proof}
Again we put $\pi=i^*\beta:h_*\mathscr{G}\to \mathscr{G}_k$.
Let $\sigma:h^*\mathscr{G}_k\isomto \mathscr{G}$ be a rigidification and
$s:\mathscr{G}_k\to h_*\mathscr{G}$ the corresponding section of $\pi$,
see Lemma~\ref{lemma:rigidifications vs sections}.
By Proposition~\ref{ker of beta}(3) the kernel of $\pi$
is $\Lie(\mathscr{G}_k)$ which is $k$-flat. The section $s$ provides
an isomorphism of $k$-schemes $\Lie(\mathscr{G}_k)\times \mathscr{G}_k\simeq h_*\mathscr{G}$
which shows that $h_*\mathscr{G}$ is $k$-flat so that
$h^*h_*\mathscr{G}$ is $k[I]$-flat. It follows also that
$\pi:h_*\mathscr{G}\to \mathscr{G}_k$ is faithfully flat and by the ``crit\`ere de
platitude par fibres'' in the nilpotent case
(\cite[\href{https://stacks.math.columbia.edu/tag/06A5}{Tag~06A5}]{SP})
we deduce that the morphism $\beta$ is faithfully flat.
Finally if $\mathscr{G}$ is of finite type over $k[I]$, then the
special fibre of $\mathrm{L}(\mathscr{G})$ is the smooth vector group
$\Lie(\mathscr{G}_k)$, hence $\mathrm{L}(\mathscr{G})$ is smooth and so is
$\beta$.
\hfill $\Box$
\end{proof}

\begin{example} \label{example:failure of surjectivity}
Here is an example where the result above fails, for a
non-rigid group. Assume $k$ is a field of characteristic
$p>0$. Let $I=k\eps$ be free of rank~1.
Let $\mathscr{G}$ be the kernel of the endomorphism
$\GG_a\to \GG_a$, $x\mapsto x^p-\eps x$. Then $(h^*h_*\mathscr{G})(R)$
is the set of elements $a\oplus b\eps \in R\oplus R\eps$
such that $(a\oplus b\eps)^p=\eps(a\oplus b\eps)$. This equation
is equivalent to $a^p=\eps a$, hence $a=0$.
Thus $(a,b)\mapsto b$ is an isomorphism
$h^*h_*\mathscr{G}\isomto\GG_{a,k[\eps]}$. The map $(h^*h_*\mathscr{G})(R)\to \mathscr{G}(R)$
sends $b$ to $b\eps_R$. In other words, if we let
$K_\eps\simeq \Spec(k[\eps][x]/(\eps x))$ denote the kernel
of $\eps:\GG_{a,k[\eps]}\to\GG_{a,k[\eps]}$,
then the sequence of the lemma is
$1 \to K_\eps \to \GG_{a,k[\eps]} \stackrel{\eps}{\too} \mathscr{G} \to 1$.
Here, the map $\eps:\GG_{a,k[\eps]}\to \mathscr{G}$ is not flat
so this is not an exact sequence of flat group schemes.
\hfill $\square$
\end{example}

\begin{prop} \label{prop:weil restriction functor}
For each $\mathscr{G}\in\Gr\!/k[I]$, the restriction via $i^*$
of the exact sequence of Lemma~\ref{lemma:beta for rigid groups}
gives $E:=h_*\mathscr{G}$ the structure of an object of
$\Ext(I)/k$. Hence Weil restriction gives a functor:
\[
h_*:\Gr\!/k[I]\to \Ext(I)/k.
\]
\end{prop}

\begin{proof}
From Lemma~\ref{lemma:beta for rigid groups} we have an exact
sequence $1 \to \mathrm{L}(\mathscr{G}) \to h^*h_*\mathscr{G} \to \mathscr{G} \to 1$.
It follows from point~(2) in Proposition~\ref{ker of beta}
that when we restrict to the closed fibre, we obtain an exact
sequence:
\[
1\too \Lie(\mathscr{G}_k,I) \too E \too \mathscr{G}_k \too 1
\]
where the $\mathscr{G}_k$-action on $\Lie(\mathscr{G}_k)$ induced
by the extension is the adjoint representation. The same
reference proves that this extension is functorial in $\mathscr{G}$.
More precisely, if $u:\mathscr{G}\to \mathscr{G}'$ is a morphism of 
affine, differentially flat, rigid $k[I]$-group schemes,
then we obtain a morphism between the extensions $E=h_*\mathscr{G}$
and $E'=h_*\mathscr{G}'$ as follows:
\[
\xymatrix{
1 \ar[r] & \Lie(\mathscr{G}_k,I) \ar[r] \ar[d]^{\diff\!\psi}
& E \ar[r] \ar[d]^{\varphi} & \mathscr{G}_k \ar[r] \ar[d]^{\psi} & 1 \\
1 \ar[r] & \Lie(\mathscr{G}'_k,I) \ar[r] & E' \ar[r] & \mathscr{G}'_k \ar[r] & 1}
\]
where $\varphi=h_*u$ and $\psi=u_k=i^*u$, the restriction
of $u$ along $i:\Spec(k)\into\Spec(k[I])$.
\hfill $\Box$
\end{proof}

We draw a corollary that will be useful in
Section~\ref{section:eq of cats}.

\begin{corollary} \label{coro1}
Let $Y$ be an affine, flat, rigid $k[I]$-scheme and
$\VV(Y)=\VV(Y/k[I])$, $\VV(Y_k)=\VV(Y_k/k)$ the vector
bundle envelopes. Then we have a split exact sequence
of flat $k$-group schemes:
\[
0 \too \VV(Y_k)\otimes_{\OO_k}\VV(I^{\smallvee})
\too h_*\VV(Y) \stackrel{\pi}{\too} \VV(Y_k) \too 0.
\]
Moreover $h_*\VV(Y)$ is flat as an $\OO_k[I]$-module scheme
and the surjection $\pi$ is isomorphic to the map given by
reduction modulo~$I$.
\end{corollary}

\begin{proof}
Recall from paragraph~\ref{defs:base ring scheme}
that $\OO_k[I]:=h_*\OO_{k[I]}$.
By the assumptions on $Y$, the vector bundle $\VV(Y)$
is affine, differentially flat and rigid over $k[I]$.
Thus Proposition~\ref{prop:weil restriction functor}
yields the displayed exact sequence. Using a rigidification
for $Y$ and base change for the functor~$\VV$, we have
$h_*\VV(Y)\simeq h_*\VV(h^*Y_k)\simeq h_*h^*\VV(Y_k)$.
Like in Proposition~\ref{prop:group algebra}(4), we have
an isomorphism of $\OO_k$-modules
$\VV(Y_k)\otimes_{\OO_k} \OO_k[I]\isomto h_*h^*\VV(Y_k)$ defined by
$x\otimes a\mapsto ax$. Given that $\VV(Y_k)$ is $\OO_k$-flat,
this proves that $h_*\VV(Y)$ is $\OO_k[I]$-flat. Even more,
we have an isomorphism
$h_*\VV(Y)\simeq\VV(Y_k)\oplus
(\VV(Y_k)\otimes_{\OO_k}\VV(I^{\smallvee}))$
and $\pi$ is the projection onto the first factor, i.e.
the map given by reduction modulo~$I$. \hfill $\Box$
\end{proof}

\section{Weil extension}
\label{section:functor tit}

In this section, we construct a functor $h^{\mbox{\tiny +}}$
called {\em Weil extension} which is a quasi-inverse
to the functor~$h_*$ of Weil restriction described in
the previous section. The idea
behind the construction is that one can recover a
$k[I]$-group scheme $\mathscr{G}$ from the extension $E=h_*\mathscr{G}$
by looking at the target of the adjunction
$\beta_{\mathscr{G}}:h^*E=h^*h_*\mathscr{G}\to\mathscr{G}$. In turn, in order to
reconstruct the faithfully flat morphism $\beta_{\mathscr{G}}$
it is enough to know its kernel $K$. In the case where
$\mathscr{G}$ is a constant group $h^*G$, which in other words is the
case where~$E$ is a
tangent bundle $\T(G,I)$, Proposition~\ref{ker of beta}(4)
hints the correct expression $K=\{(x,g) \in h^*E;\ g=\exp(-x)\}$.
The definition of $K$ for general extensions
$1\to \Lie(G,I) \to E \to G \to 1$
where $G$ is an affine, differentially flat $k$-group scheme,
builds on this intuition.

\subsection{Hochschild extensions}

The construction of an extension from a 2-cocycle is
well-known; we recall it to set up the notations.
Recall from \cite[chap.~II, \S~3, no~2]{DG70} that if $G$ is
a $k$-group functor and $M$ is a $k$-$G$-module functor,
then a {\em Hochschild extension} or simply {\em $H$-extension}
of $G$ by $M$ is an exact sequence of group functors
\[
1\too M\stackrel{i}{\too} E\stackrel{\pi}{\too} G
\]
such that $\pi$ has a section (which is not required to
be a morphism of groups). From a given section $s:G\to E$,
we can produce a unique morphism $c:G\times G\to M$ such that
$i(c(g,g')):=s(g)s(g')s(gg')^{-1}$.
This is a 2-cocycle, i.e. it satisfies the identity
\[
c(g,g')+c(gg',g'')=g\cdot c(g',g'')+c(g,g'g'').
\]
Note that we may always replace $s$ by the section
$G\to E$, $g\mapsto s(1)^{-1}s(g)$ to obtain a section
such that $s(1)=1$. When this is the
case, we have $c(g,1)=c(1,g')=0$ for all $g,g'$ and we say
that $c$ is {\em normalized}.
Conversely, starting from a cocycle $c$, the functor
$E_c=M\times G$ with multiplication defined by
\[
(x,g)\cdot (x',g'):=\big(x+g\cdot x'+c(g,g'),gg'\big)
\]
is an $H$-extension. 
The map $s:G\to E_c$, $g\mapsto (0,g)$ is a possible
choice of section for $\pi$. It follows from the previous
comments that we may always change the cocycle into a
normalized cocycle.

\subsection{Kernel of the adjunction, reprise}
\label{cocycles}

In this subsection, we prepare the construction of the
kernel of the adjunction map $\beta_{h^{\mbox{\tiny +}}E}$
of the (yet to be produced) Weil extension
$h^{\mbox{\tiny +}}E$.
The end result is in Proposition~\ref{prop:subgroup K}
of the next subsection. Note that in spite of the similarity
of titles, the viewpoint is different from that of
Subsection~\ref{subsection:the map beta}.

Let~$G$ be an affine $k$-group scheme, and $\Lie(G,I)$
its Lie algebra relative to $I$, viewed as an affine
$k$-group scheme with the adjoint action of $G$. To any
2-cocycle $c:G\times G\too \Lie(G,I)$ we attach as before
an $H$-extension $E_c=\Lie(G,I)\times G$ with multiplication:
\[
(x,g)\cdot (x',g'):=\big(x+\Ad(g)x'+c(g,g'),gg'\big).
\]
Our group $E_c$ has a structure of $H$-extension:
\[
\xymatrix@C=6mm{
1 \ar[r] & \Lie(G,I) \ar[rr]^-{x\mapsto (x,1)} & &
E_c \ar[rr]^-{(x,g)\mapsto g} & & G \ar[r] & 1.}
\]

The following result is the heart of the construction of
the Weil extension functor $h^{\mbox{\tiny +}}$. We point out
that among the groups $K_\lambda(E_c)$ introduced here, it is
especially $K_{-1}(E_c)$ that will be relevant in the sequel,
as Proposition~\ref{ker of beta}(4) shows. However, we include the
whole family $K_\lambda(E_c)$ since it comes with no extra cost
and brings interesting insight, in the sense that it ultimately
provides an explicit linear path in the $\OO_k$-module stack
$\Gr/k[I]$.

\begin{prop} \label{prop:sous-groupe Klambda(E_c)}
Let $h:\Spec(k[I])\to\Spec(k)$ be the $k$-scheme of dual
numbers. Fix $\lambda \in k$.
\begin{trivlist}
\itemn{1} Let $G$ be an affine $k$-group scheme and let $E_c$ be the
$H$-extension constructed out of a normalized 2-cocycle
$c:G\times G\too \Lie(G,I)$.
Let $K_\lambda(E_c)\subset h^*E_c$ be the subfunctor defined by:
\[
K_\lambda(E_c)=\{ (x,g) \in h^*E_c\,;\,g=\exp(\lambda x)\}.
\]
Then $K_\lambda(E_c)$ is a closed normal sub-$k[I]$-group
scheme of $h^*E_c$.
\itemn{2} Let $G,G'$ be affine $k$-group schemes and
$E_c,E_{c'}$ be the $H$-extensions constructed out of some
chosen normalized 2-cocycles $c,c'$. Let $f:E_c\to E_{c'}$
be a morphism of extensions:
\[
\xymatrix{
1 \ar[r] & \Lie(G,I) \ar[r] \ar[d]^{\Lie(\alpha)} &
E_c \ar[r] \ar[d]^f & G \ar[r] \ar[d]^{\alpha} & 1 \\
1 \ar[r] & \Lie(G',I) \ar[r] & E_{c'} \ar[r] & G' \ar[r] & 1}
\]
Then $(h^*f)(K_\lambda(E_c))\subset K_\lambda(E_{c'})$,
with equality if $f$ is an isomorphism.
\end{trivlist}
\end{prop}

\bigskip

When the extension $E_c$ is clear from context, we write
$K_\lambda$ instead of $K_\lambda(E_c)$.
We will prove the proposition after a few preliminaries.
First of all, for the convenience of the reader, we recall
the description of morphisms of extensions, in the abstract
group setting for simplicity.

\begin{lemma} \label{lemma:morphisms of extensions}
Let $\alpha:G\to G'$ be a morphism of groups and
$\delta:L\to L'$ be a morphism from a $G$-module to a
$G'$-module which is $\alpha$-equivariant.
Let $E\in \cE xt(G,L)$ and $E'\in \cE xt(G',L')$ be two
extensions.
\begin{trivlist}
\itemn{1} There exists a morphism of extensions
$f:E\to E'$, i.e. a diagram
\[
\xymatrix{
1 \ar[r] & L \ar[r] \ar[d]^{\delta} &
E \ar[r] \ar@{..>}[d]^f & G \ar[r] \ar[d]^{\alpha} & 1 \\
1 \ar[r] & L' \ar[r] & E' \ar[r] & G' \ar[r] & 1,}
\]
if and only if $\alpha^*[E']=\delta_*[E]$
in $\mathrm{H}^2(G,L')$, and if this condition holds then the set
of morphisms is a principal homogeneous space under
the set of 1-cocycles $\mathrm{Z}^1(G,L')$. More precisely, assume
that we describe $E$ with a normalized
cocycle $c:G\times G\to L$ so that $E\simeq L\times G$ with
multiplication $(x,g)\cdot (x',g')=(x+g\cdot x'+c(g,g'),gg')$,
and we describe $E'$ similarly with a normalized cocycle $c'$.
Then all morphisms $f:E\to E'$ are of the form
$f(x,g)=(\delta(x)+\varphi(g),\alpha(g))$
for a unique 1-cochain $\varphi:G\to L'$ such that
$\partial \varphi=c'\circ\alpha-\delta\circ c$.
\itemn{2} If $E,E'$ are two extensions of $G$ by $L$,
then the set of morphisms of extensions $E\to E'$ is a
principal homogeneous space under the group $\mathrm{Z}^1(G,L)$,
more precisely all morphisms are of the form
$f(x,g)=(x+\varphi(g),g)$ for a unique $\varphi\in \mathrm{Z}^1(G,L)$.
All of them are isomorphisms.
\itemn{3} Assume that the extension is trivial, so that
$[E]=0\in \mathrm{H}^2(G,L)$. Then all group sections $G\to E$ of
the extension are of the form $s(x,g)=(\varphi(g),g)$ for a
unique $\varphi\in \mathrm{Z}^1(G,L)$ such that $\partial \varphi=c$.
\end{trivlist}
\end{lemma}

\begin{proof}
Any morphism of extensions can be written  as
$f:L\times G\to L'\times G'$,
$(x,g)\mapsto (u(x,g),\alpha(g))$ with $u(x,1)=\delta(x)$.
The property that $f$ is a morphism of
groups translates into the identity:
\[
u(x_1+g_1\cdot x_2+c(g_1,g_2),g_1g_2)=
u(x_1,g_1)+\alpha(g_1)\cdot u(x_2,g_2)+c'(\alpha g_1,\alpha g_2).
\]
Setting $x_1=x$, $x_2=0$, $g_1=1$, $g_2=g$, and
$\varphi(g):=u(0,g)$ we find $u(x,g)=\delta(x)+\varphi(g)$
for all $x,g$. The above identity implies
$\varphi(g_1g_2)-\varphi(g_1)-\alpha(g_1)\cdot\varphi(g_2)
=c'(\alpha g_1,\alpha g_2)-\delta(c(g_1,g_2))$.
This means that $\partial \varphi=c'\circ\alpha-\delta\circ c$
as claimed in~(1).
Considering the particular case of morphisms
\[
\xymatrix{
1 \ar[r] & L \ar[r] \ar[d]^{\id} &
E \ar[r] \ar@{..>}[d] & G \ar[r] \ar[d]^{\id} & 1 \\
1 \ar[r] & L \ar[r] & E' \ar[r] & G \ar[r] & 1}
\]
we get~(2), and considering the case of morphisms
\[
\xymatrix{
1 \ar[r] & 0 \ar[r] \ar[d]^{0} &
G \ar[r]^{\id} \ar@{..>}[d] & G \ar[r] \ar[d]^{\id} & 1 \\
1 \ar[r] & L \ar[r] & E \ar[r] & G \ar[r] & 1}
\]
we get~(3).
\hfill $\Box$
\end{proof}

We come back to the extension $E_c$. The lemma tells us that
the group $\Aut_{\ext}(E_c)$ of automorphisms of~$E_c$ as an
extension is isomorphic to the group of 1-cocycles
$\mathrm{Z}^1(G,\Lie(G,I))$. Item (2) of
Proposition~\ref{prop:sous-groupe Klambda(E_c)} says in particular
that $K_\lambda(E_c)$ is stable under these particular automorphisms.

Now we record a few technical properties concerning the
exponential and the cocycles. For simplicity we write
$\exp$ instead of $\exp_G$.

\begin{lemma} \label{lemma:properties of cocycle}
Let $G$ be an affine $k$-group scheme and
$c:G\times G\too \Lie(G,I)$ a normalized 2-cocycle.
Let $\exp:h^*\Lie(G,I)\to h^* G$ be the exponential morphism
as defined in Subsection~\ref{subsection:exp and inf translation}.
Let $R$ be a $k[I]$-algebra and $g,g',g''\in G(R)$. Assume
that $g$ is an exponential i.e. an element in the image of
$\exp$. Then, we have:
\begin{trivlist}
\itemn{1} $\exp(c(g,g'))=\exp(c(g',g))=1$,
\itemn{2} 
$\exp(c(gg',g''))=\exp(c(g'g,g''))
=\exp(c(g',gg''))=\exp(c(g',g''g))
=\exp(c(g',g''))$.
\end{trivlist}
The same statements hold with $c$ replaced by $\lambda c$,
for each $\lambda \in k$.
\end{lemma}

\begin{proof}
(1) Apply Lemma~\ref{lemma:cocycle vanishes with eps} to the
morphism of pointed schemes
$\phi=c(-,g'):G\otimes_k R\to (\Lie(G,I))\otimes_k R$ to obtain $\exp(c(g,g'))=1$. Similarly $\exp(c(g',g))=1$.

\smallskip

\noindent (2) Write $g=\exp(x)$. Since
$\Lie(G,I)=\Lie(G)\otimes\VV(I^\smallvee)$ we can write $x$
as a sum of tensors $y\otimes i$. Working inductively on the
number of tensors in the sum, we can assume that $x=y\otimes i$.
We prove successively that
each of the first four terms equals $\exp(c(g',g''))$.
\begin{enumerate}
\item[a.] The \ cocycle \ identity \
\ $c(g,g')\,+\,c(gg',g'')\,=\,\Ad(g)\,c(g',g'')\,+\,c(g,g'g'')$ \ together \ with \ (1) \
imply \ $\exp(c(gg',g''))=\exp(\Ad(g)c(g',g''))$.
Since $g$ is an exponential, according
to Proposition~\ref{prop:elementary properties of exp}(4)
its adjoint action
is given by $\Ad(g)c(g',g'')=c(g',g'')+i[x,c(g',g'')]$.
Since multiples of $i$ lie in the kernel of $\exp$,
see Proposition~\ref{prop:elementary properties of exp}(6),
we deduce $\exp(\Ad(g)c(g',g''))=\exp(c(g',g''))$.
\item[b.] Since $\gamma:=g'g(g')^{-1}$ is an exponential,
$\exp(c(g'g,g''))=\exp(c(\gamma g',g''))=\exp(c(g',g''))$
by a.
\item[c.] The cocycle identity with $g$ and $g'$ exchanged reads
$c(g',g)+c(g'g,g'')=\Ad(g')c(g,g'')+c(g',gg'')$. We deduce
$\exp(c(g'g,g''))=\exp(c(g',gg''))$. We conclude with b.
\item[d.] Again this follows from the fact that
$g''g(g'')^{-1}$ is an exponential.
\end{enumerate}
The final claim of the lemma holds because $\lambda c$ is again
a normalized cocycle.
\hfill $\Box$
\end{proof}

\begin{proof-of}{Proposition~\ref{prop:sous-groupe Klambda(E_c)}}
Let us write $K=K_\lambda(E_c)$ for simplicity.
Obviously $K$ is a closed subfunctor of $h^*E_c$
which is isomorphic to $h^*\Lie(G,I)$ as a $k[I]$-scheme.
For the verification of points (1) and (2) we let $R$ be
an arbitrary $k[I]$-algebra.

\medskip

\noindent (1) First, let us prove that $K$
is a subgroup scheme. Let $(x,g)$ and $(x',g')$ be two
$R$-valued points of $K$ so we have
$g=\exp(\lambda x)$ and $g'=\exp(\lambda x')$. On one hand,
using Proposition~\ref{prop:elementary properties of exp}(4)
we see that $\Ad(g)x'=x'+\eps [x,x']$ and by
Proposition~\ref{prop:elementary properties of exp}(6)
we deduce $\exp(\lambda \Ad(g)x')=\exp(\lambda x')$.
On the other hand,
by Lemma~\ref{lemma:properties of cocycle}(1) we have
$\exp(\lambda c(g,g'))=1$. Putting all this together we get:
\begin{align*}
\exp\big(\lambda x+\lambda \Ad(g)x'+\lambda c(g,g')\big)
& =\exp(\lambda x)\, \exp(\lambda \Ad(g)x')\,\exp(\lambda c(g,g')) \\
& =\exp(\lambda x)\,\exp(\lambda x') \\
& =gg'.
\end{align*}
This proves that the product $(x,g)\cdot (x',g')$ is a point
of $K$. Using the same arguments we prove that the inverse
$(x,g)^{-1}=(-\Ad(g^{-1})x-c(g^{-1},g),g^{-1})$ is a point
of $K$. Hence $K$ is a subgroup scheme.

\smallskip

Second, let us prove that $K$ is stable by inner
automorphisms. Let $(x,g)$ and $(x',g')$ be $R$-valued points
of $h^*E_c$ and $K$ respectively. We must prove that
$(x'',g''):=(x,g)\cdot(x',g')\cdot(x,g)^{-1}$ lies in $K$.
Writing $x'$ as a sum of tensors $x_s=y_s\otimes i_s$ and
setting $g_s=\exp(\lambda x_s)$, we have $(x_s,g_s)\in K(R)$
and since~$K$ is a subgroup scheme,
it is enough to prove that $(x,g)\cdot(x_s,g_s)\cdot(x,g)^{-1}$
lies in $K$. In other words, we may and do assume in the sequel
that $x'=y\otimes i$.
We first consider $(x_1,g_1):=(x',g')\cdot(x,g)^{-1}$.
Using the fact that $g'=\exp(\lambda x')$ and
Proposition~\ref{prop:elementary properties of exp}(4), we find
\[
\Ad(g')\big(-\Ad(g^{-1})x-c(g^{-1},g)\big)
=-\Ad(g^{-1})x-c(g^{-1},g)+b
\]
where $b\in I\!\cdot\!\Lie(G,I)(R)$ is a certain bracket, and hence:
\begin{align*}
(x_1,g_1) & =(x',g')\cdot(-\Ad(g^{-1})x-c(g^{-1},g),g^{-1}) \\
& = \big(x'-\Ad(g^{-1})x-c(g^{-1},g)+b+c(g',g^{-1}),g'g^{-1}\big).
\end{align*}
Now
$(x'',g'')=(x,g)\cdot (x_1,g_1)=(x+\Ad(g)x_1+c(g,g_1),gg_1)$
and our task is to check that
\[
\exp\big(\lambda x+\lambda \Ad(g)x_1+\lambda c(g,g_1)\big)=gg_1.
\]
We note the following:
\begin{enumerate}
\item[a.] We have: \
$\lambda x+\lambda \Ad(g)x_1=\lambda \Ad(g)x'
-\lambda \Ad(g)c(g^{-1},g)+b+\lambda \Ad(g)c(g',g^{-1})$. \
By Proposition \ref{prop:elementary properties of exp}(6),
the term $b$ will disappear upon taking exponentials,
so we may disregard it. Similarly,
by Lemma~\ref{lemma:properties of cocycle}(1) the exponential
of $\lambda \Ad(g)c(g',g^{-1})$ equals 1. Also, using the
cocycle relation we see
that $\Ad(g)c(g^{-1},g)=c(g,g^{-1})$. Hence:
\[
\exp(\lambda x+\lambda \Ad(g)x_1)
=\exp\big(\lambda \Ad(g)x'-\lambda c(g,g^{-1})\big).
\]
\item[b.] By Lemma~\ref{lemma:properties of cocycle}(2) we have
$\exp(\lambda c(g,g_1))=\exp(\lambda c(g,g'g^{-1}))
=\exp(\lambda c(g,g^{-1}))$.
\item[c.] Using Proposition~\ref{prop:elementary properties of exp}(3)
we have
$gg_1=gg'g^{-1}=g\exp(\lambda x')g^{-1}=\exp(\lambda \Ad(g)x')$.
\end{enumerate}
Putting a-b-c together we get
$\exp\big(\lambda x+\lambda \Ad(g)x_1+\lambda c(g,g_1)\big)=
\exp(\lambda \Ad(g)x')=gg_1$
as desired.

\smallskip

\noindent (2) Let us write $K=K_\lambda(E_c)$ and
$K'=K_\lambda(E_{c'})$ for simplicity. According to
Lemma~\ref{lemma:morphisms of extensions},
any morphism of extensions $f:E_c\to E_{c'}$ is of the form
$f(x,g)=(\Lie(\alpha)(x)+\varphi(g),\alpha(g))$ for a unique
$\varphi:G\to \Lie G'$ satisfying
$\varphi(gg')-\varphi(g)-\Ad(g)\varphi(g')
=c'(\alpha g,\alpha g')-\Lie(\alpha)(c(g,g'))$. Setting $g=g'=1$
we see that such a $\varphi$, hence also $\lambda \varphi$,
is a map of pointed schemes. This being said, if $(x,g)$
is an $R$-valued point of $K$, the following computation
shows that $f(x,g)$ is a point of $K'$:
\begin{align*}
\exp_{G'}(\lambda \Lie(\alpha)(x)+\lambda \varphi(g))
& =\exp_{G'}(\lambda \Lie(\alpha)(x))
\exp_{G'}(\lambda \varphi(g)) \\
& =\exp_{G'}(\lambda \Lie(\alpha)(x))
\ \mbox{by Lemma~\ref{lemma:cocycle vanishes with eps},} \\
& =\exp_{G'}(\Lie(\alpha)(\lambda x)) \\
& = \alpha(\exp_{G}(\lambda x))
\ \mbox{by functoriality of $\exp$,} \\
& =\alpha(g) \ \mbox{because $g=\exp_G(\lambda x)$}.
\end{align*}
When $f$ is an isomorphism, applying the statement to
$f^{-1}$, we find $(h^*f)(K)=K'$.
\hfill $\Box$
\end{proof-of}

\subsection{Weil extension functor}

Now let $G$ be an affine and differentially flat $k$-group
scheme. Thus $G$ as well as the adjoint representation
$\Lie(G,I)$ are $k$-flat. We consider an arbitrary extension:
\[
1\too \Lie(G,I)\stackrel{i}{\too} E \stackrel{\pi}{\too} G \too 1.
\]
Then $E\to G$ is an fpqc torsor under $\Lie(G,I)$.
It has a cohomology class in $\mathrm{H}^1(G,\Lie(G,I))$ which
vanishes, being quasi-coherent cohomology of an affine scheme.
It follows that $\pi$ has a section $s:G\to E$, and the
extension becomes an $H$-extension. We may and do replace
$s$ by $s(1)^{-1}\cdot s$ in order to ensure that $s(1)=1$.
From $s$ we build a normalized cocycle $c:G\times G\too \Lie(G,I)$
as follows:
\[
i(c(g,g')):=s(g)s(g')s(gg')^{-1}.
\]
These data give rise to the group $E_c$ as defined
in Subsection~\ref{cocycles}.

\begin{prop} \label{prop:subgroup K}
Let $1\to \Lie(G,I) \to E \to G \to 1$ be an object of the
category $\Ext(I)/k$. Let $s:G\to E$,
with $s(1)=1$, be as chosen above, and $c$ the normalized
cocycle derived from it. Let $\lambda \in k$.
\begin{trivlist}
\itemn{1} The map $\tau_s:E_c\to E$, $(x,g)\mapsto i(x)s(g)$
is an isomorphism of extensions.
\itemn{2} The closed normal subgroup scheme
$K_\lambda(E):=(h^*\tau_s)(K_\lambda(E_c)) \subset h^*E$
does not depend on the choice of~$s$.
\itemn{3} For all morphisms $f:E\to E'$ in $\Ext(I)/k$
we have $(h^*f)(K_\lambda(E))\subset K_\lambda(E')$, with equality
if $f$ is an isomorphism.
\end{trivlist}
\end{prop}

If the extension $E$ is clear from context, we write $K_\lambda$
instead of $K_\lambda(E)$. Note that if $E$ is the trivial
extension and $s=\alpha$, the map $\tau_\alpha$ is the map
$\varrho_G$ defined in paragraph~\ref{TG and Lie G}.

\begin{proof}
(1) follows from the constructions of $c$ and $E_c$.

\smallskip

\noindent For the proof of (2) and (3) we will rely on
the following basic remark. Let $f:E\to E'$ be a morphism
in $\Ext(I)/k$. Let $\tau_s:E_c\to E$ and $\tau_{s'}:E_{c'}\to E'$
be the isomorphisms associated to choices of sections $s,s'$
preserving $1$ and corresponding normalized cocycles $c,c'$.
Let $K_{\lambda,s}(E):=(h^*\tau_s)(K_\lambda(E_c)) \subset h^*E$ and similarly
$K_{\lambda,s'}(E'):=(h^*\tau_{s'})(K_\lambda(E_{c'})) \subset h^*E'$.
We have a morphism of extensions:
\[
\rho=\tau_{s'}^{-1}\circ f\circ\tau_s:E_c \too E_{c'}.
\]
According to Proposition~\ref{prop:sous-groupe Klambda(E_c)}(2) we have
$(h^*\rho)(K_\lambda(E_c))\subset K_\lambda(E_{c'})$. It follows that:
\[
(h^*f)(K_{\lambda,s}(E))=(h^*f)((h^*\tau_s)(K_\lambda(E_c)))
=(h^*\tau_{s'})((h^*\rho)(K_\lambda(E_c)))
\subset (h^*\tau_{s'})(K_\lambda(E_{c'}))=K_{\lambda,s'}(E').
\]
When $f$ is an isomorphism, applying the statement to
$f^{-1}$ gives equality.

\smallskip

\noindent (2) Applying the basic remark to $E=E'$ and
$f=\id:E\to E$ proves that $K_{\lambda,s}(E)=K_{\lambda,s'}(E)$,
that is, the subgroup $K_{\lambda,s}(E)$ does not depend on the
choice of~$s$.
Since $\tau_s$ is an isomorphism of groups, the fact
that $K_\lambda(E)$ is a closed normal subgroup scheme
follows from Proposition~\ref{prop:sous-groupe Klambda(E_c)}(1).

\smallskip

\noindent (3) Applying the basic remark to a general $f$
gives the statement.
\hfill $\Box$
\end{proof}

For an extension $1\to \Lie(G) \to E \to G \to 1$,
we let $K_\lambda:=K_\lambda(E)$ be the normal subgroup defined
in the proposition. Point~(4) in Proposition~\ref{ker of beta}
gives motivation to consider $K_{-1}$.
The fpqc quotient sheaf $h^{\mbox{\tiny +}}E:=h^*E/K_{-1}$ is
representable by an affine flat
$k[I]$-scheme (see Perrin~\cite[Cor.~0.2]{Per76}).

\begin{definition}
We call {\em Weil extension} the quotient
$h^{\mbox{\tiny +}}E:=h^*E/K_{-1}$.
\end{definition}

\begin{lemma}
Weil extension is a functor $\Ext(I)/k\to \Gr\!/k[I]$.
\end{lemma}

\begin{proof}
The $k[I]$-group scheme $G:=h^*E/K_{-1}$ is affine and flat.
Let $s:G\to E$ be a section of $E\to G$
such that $s(1)=1$.
By pullback, this induces a morphism $h^*G\to h^*E\to G$
which is the identity on the special fibre, hence an
isomorphism, hence a rigidification. This proves that the
functor of the statement is well-defined on objets.
Proposition~\ref{prop:subgroup K} proves that the functor
is well-defined on morphisms.
\hfill $\Box$
\end{proof}

\section{The equivalence of categories}
\label{section:eq of cats}

This section is devoted to the proof of
Theorem~\ref{main_theorem_0}, which we recall below for
ease of reading. The plan is as follows.
In Subsection~\ref{subsection:equivariance} we
prove a preliminary result used in the proof of~(1).
In Subsection~\ref{section:proof} we prove (1), (2), (4).
Finally in Subsection~\ref{subsection:iso of module stacks}
we prove (3).

\begin{theorem} \label{main_theorem}
{\rm (1)} The Weil restriction/extension functors
provide quasi-inverse equivalences:
\[
\xymatrix@C=12mm{
\Gr\!/k[I]\ \ar@<.6ex>[r]^-{h_*}
& \ \Ext(I)/k \ar@<.6ex>[l]^-{h^{\mbox{\rm\tiny +}}}.}
\]
These equivalences commute with base
changes $\Spec(k')\to\Spec(k)$.

\smallskip

\noindent {\rm (2)} If $1\to\mathscr{G}'\to\mathscr{G}\to\mathscr{G}''\to 1$ is an
exact sequence in $\Gr\!/k[I]$, then
$1\to h_*\mathscr{G}'\to h_*\mathscr{G}\to h_*\mathscr{G}''$ is exact in $\Ext(I)/k$.
If moreover $\mathscr{G}'$ is smooth then
$1\to h_*\mathscr{G}'\to h_*\mathscr{G}\to h_*\mathscr{G}''\to 1$ is exact. In particular,
$h_*$ is an exact equivalence between the subcategories
of smooth objects endowed with their natural exact structure.

\smallskip

\noindent {\rm (3)} The equivalence $h_*$ is a morphism
of $\OO_k$-module stacks fibred over $\Gr\!/k$, i.e.
it transforms
the addition and scalar multiplication of deformations
of a fixed $G\in\Gr\!/k$ into the Baer sum and scalar
multiplication of extensions of $G$ by $\Lie(G,I)$.

\smallskip

\noindent {\rm (4)} Let $P$ be one of the properties:
of finite type, smooth, connected, unipotent, split unipotent,
solvable, commutative. Then $\mathscr{G}\in \Gr\!/k[I]$ has the
property $P$ if and only if the $k$-group scheme $E=h_*\mathscr{G}$
has $P$.
\end{theorem}

\subsection{Equivariance of rigidifications under
Lie algebra translation}
\label{subsection:equivariance}

Let $\mathscr{G}$ be an affine, differentially flat, rigid $k[I]$-group
scheme. Let $\sigma:h^*\mathscr{G}_k\isomto \mathscr{G}$ be a rigidification
such that $\sigma(1)=1$. We consider the morphism of $k$-schemes:
\[
h_*\sigma:h_*h^*\mathscr{G}_k \too h_*\mathscr{G}.
\]
This is not a morphism of group schemes, because source and
target are not isomorphic groups in general. However, it
satisfies an important equivariance property. To state it,
note that source and target are extensions of $\mathscr{G}_k$ by
$\Lie(\mathscr{G}_k,I)$; in particular both carry an action of
$\Lie(\mathscr{G}_k,I)$ by left translation.

\begin{prop} \label{prop:sigma is Lie-equivariant}
With notation as above, the morphism of $k$-schemes
\[
h_*\sigma:h_*h^*\mathscr{G}_k \too h_*\mathscr{G}
\]
is $\Lie(\mathscr{G}_k,I)$-equivariant.
\end{prop}

\begin{proof}
We write simply $\OO$ instead of $\OO_k$ or $\OO_{k[I]}$
when the base is clear from context.
Consider the extension of $\sigma$ to the group algebras:
\[
\sigma':=h_*\OO[\sigma]\ :h_*\OO[h^*\mathscr{G}_k]\to h_*\OO[\mathscr{G}].
\]
Note that by compatibility of $\OO[-]$ with base change and
Weil restriction (see Proposition~\ref{prop:group algebra},
(2)-(4)), we have
$h_*\OO[h^*\mathscr{G}_k] \isomto h_*h^*\OO[\mathscr{G}_k]
\isomto \OO[\mathscr{G}_k][I]$.
We obtain a commutative diagram:
\[
\xymatrix@C=12mm{
\underset{\ }{h_*h^*\mathscr{G}_k} \ar[r]^-{h_*\sigma} \ar@{^(->}[d]
& \underset{\ }{h_*\mathscr{G}} \ar@{^(->}[d] \\
\overset{\ }{\OO[\mathscr{G}_k][I]} \ar[r]^-{\sigma'}
& \overset{\ }{h_*\OO[\mathscr{G}]}.}
\]
We identify
$I\cdot\OO[\mathscr{G}_k]:=\OO[\mathscr{G}_k]\otimes_{\OO_k}\VV(I^{\smallvee})$
as the ideal generated
by $I$ inside both algebras $\OO[\mathscr{G}_k][I]$ and
$h_*\OO[\mathscr{G}]$, see Corollary~\ref{coro1}. With this
convention we formulate:

\bigskip

\noindent {\bf Claim.}
{\em We have $\sigma'(y+x)=\sigma'(y)+x$
for all points $y\in\OO[\mathscr{G}_k][I]$ and
$x\in I\cdot\OO[\mathscr{G}_k]$.}

\bigskip

\noindent To prove this, we introduce another copy $J=I$ of
our square-zero ideal as follows:
\[
\xymatrix@C=16mm{
h_*h^*h_*h^*\mathscr{G}_k \ar[r]^-{h_*h^*h_*\sigma} \ar[d]
& h_*h^*h_*\mathscr{G} \ar[d] \\
\OO[\mathscr{G}_k][I][J] \ar[r]^-{\sigma''}
& h_*\OO[\mathscr{G}][J]}
\]
where we have set $\sigma''=h_*h^*\sigma'$ for brevity.
Let $s,t:\OO[\mathscr{G}_k][I]\times J\cdot\OO[\mathscr{G}_k][I][J]
\to h_*\OO[\mathscr{G}][J]$ be defined by
\[
s(y,x)=\sigma''(y+x)
\quad\mbox{and}\quad
t(y,x)=\sigma''(y)+x.
\]
Then $s$ and $t$ are equal modulo $I$
because of the fact that $\sigma$ is the identity
modulo $I$. Clearly they are also equal modulo $J$.
Since $h_*\OO[\mathscr{G}][J]$ is flat as an $\OO_k[I][J]$-module,
we deduce that $s-t$ takes its values in the ideal
$IJ\cdot h_*\OO[\mathscr{G}][I]$. Setting $J=I$, that is composing with
the morphism $h_*\OO[\mathscr{G}][J]\to h_*\OO[\mathscr{G}]$ that takes a
section of $J$ to the corresponding section of $I$,
we obtain the claim.

\bigskip

\noindent We now conclude the proof of the proposition.
We write $\star_1$ and $\star_2$ the multiplications
of $h_*h^*\mathscr{G}_k$ and $h_*\mathscr{G}$ respectively, extended
to $h_*\OO[h^*\mathscr{G}_k]$ and $h_*\OO[\mathscr{G}]$. It is enough
to show more generally that
\[
h_*\sigma:h_*\OO[h^*\mathscr{G}_k] \too h_*\OO[\mathscr{G}]
\]
is $(1+I\cdot\OO[\mathscr{G}_k])$-equivariant. We take points
$x\in I\cdot\OO[\mathscr{G}_k]$ and $y\in h_*\OO[h^*\mathscr{G}_k]$
and compute:
\[
\setlength{\arraycolsep}{.6mm}
\begin{array}{rll}
(h_*\sigma)((1+x)\star_1 y)
& = (h_*\sigma)(y+x\star_1 y) \\
& = (h_*\sigma)(y)+x\star_1 y
\ \mbox{ by the Claim above,} \\
& = (h_*\sigma)(y)+x\star_2 y
\ \mbox{ by Corollary~\ref{coro1} since $\star_1=\star_2$
modulo $I$,} \\
& = (h_*\sigma)(y)+x\star_2 (h_*\sigma)(y)
\ \mbox{ by Corollary~\ref{coro1} since $\sigma=\id$
modulo $I$,} \\
& = (1+x)\star_2 (h_*\sigma)(y).
\end{array}
\]
This proves that $h_*\sigma$ is $\Lie(\mathscr{G}_k,I)$-equivariant.
\hfill $\Box$
\end{proof}

\subsection{Proof of the main theorem: equivalence and exactness}
\label{section:proof}

\begin{item-title}{Proof of~\ref{main_theorem}(1)}
We shall prove that the functors
\[
\xymatrix@C=12mm{
\Gr\!/k[I]\ \ar@<.6ex>[r]^-{h_*}
& \ \Ext(I)/k \ar@<.6ex>[l]^-{h^{\mbox{\rm\tiny +}}}}
\]
provide quasi-inverse equivalences that commute with base
changes $\Spec(k')\to\Spec(k)$.
Firstly, we prove that $h^{\mbox{\rm\tiny +}}\circ h_*$ is
isomorphic to the identity. Let $\mathscr{G}\to\Spec(k[I])$ be an
affine, differentially flat, rigid $k[I]$-group scheme.
Let $E=h_*\mathscr{G}$ be the associated extension:
\[
1\too \Lie(\mathscr{G}_k,I) \too E \too \mathscr{G}_k \too 1.
\]
We fix a rigidification $\sigma:h^*\mathscr{G}_k\isomto \mathscr{G}$ such
that $\sigma(1)=1$. We know from
Proposition~\ref{prop:sigma is Lie-equivariant} that the map
$h_*\sigma:h_*h^*\mathscr{G}_k \too h_*\mathscr{G}$ is $\Lie(\mathscr{G}_k,I)$-equivariant.
If we use the letter $\gamma$ to denote the inclusions of
$\Lie(\mathscr{G}_k,I)$ into the relevant extensions, this can be written:
\[
(h_*\sigma)(\gamma_{h_*h^*\mathscr{G}_k}(x)\cdot y)
=\gamma_{h_*\mathscr{G}}(x)\cdot (h_*\sigma)(y),
\quad\mbox{all } x\in\Lie(\mathscr{G}_k,I), \ y\in h_*h^*\mathscr{G}_k.
\]
Restricting to $y$ in the image of
$\alpha=\alpha_{\mathscr{G}_k}:\mathscr{G}_k\into h_*h^*\mathscr{G}_k$, so
$\tau_\alpha=\varrho_{\mathscr{G}_k}$, we obtain:
\[
h_*\sigma\circ \tau_\alpha=\tau_s.
\]
Using functoriality of $\beta$ and the fact that
$\sigma(1)=1$, we build a commutative diagram:
\[
\xymatrix@C=15mm{
h^*F=h^*(\Lie(\mathscr{G}_k,I)\times \mathscr{G}_k)
\ar[r]^-{h^*\tau_\alpha} \ar[d]_{\id}
& h^*h_*h^* \mathscr{G}_k \ar[r]^-{\beta_{h^*\mathscr{G}_k}}
\ar[d]_-{h^*h_*\sigma} & h^*\mathscr{G}_k \ar[d]_-{\sigma} \ar[r] & 1 \\
h^*E_c=h^*(\Lie(\mathscr{G}_k,I)\times \mathscr{G}_k) \ar[r]^-{h^*\tau_s}
& h^*h_*\mathscr{G} \ar[r]^-{\beta_\mathscr{G}} & \mathscr{G} \ar[r] & 1}
\]
Here the horizontal maps are morphisms of groups and the vertical
maps are {\em not} morphisms of groups (not even the
leftmost map $\id:h^*F\to h^*E_c$).
Note also that $F$ is $E_0=\T(\mathscr{G}_k,I)$, that is, the extension
$E_c$ with the zero cocycle $c=0$. Now we consider
$K_{-1}(E)$ as defined in Proposition~\ref{prop:subgroup K}.
According to Proposition~\ref{ker of beta}(4), the group
$K_{-1}(E_0)\subset E_0$ is the kernel of
$\beta_{h^*\mathscr{G}_k}\circ h^*\tau_\alpha$. On one hand the identity
takes $K_{-1}(E_0)$ to $K_{-1}(E_c)$, and on the other hand the map
$h_*\sigma$ takes $\ker(\beta_{h^*\mathscr{G}_k})$ onto $\ker(\beta_\mathscr{G})$
since it takes~1 to~1. By commutativity of the left-hand
square, we find $K_{-1}(E)=\ker(\beta_\mathscr{G})$ and thefore
$\beta_\mathscr{G}$ induces an isomorphism
$h^{\mbox{\rm\tiny +}}E=h^*E/K_{-1}(E)\simeq \mathscr{G}$ which is
visibly functorial.

\medskip

Secondly, we prove that $h_* \circ h^{\mbox{\rm\tiny +}}$
is isomorphic to the identity.
Let $1\to \Lie(G,I) \to E \to G \to 1$ be an extension.
We fix a section $s:G\to E$ such that $s(1)=1$
and we let $c$ be the normalized cocycle defined by $s$.
Let $K_{-1}=K_{-1}(E)\subset h^*E$ be the closed normal subgroup
defined in Proposition~\ref{prop:subgroup K}, and let
$\mathscr{G}:=h^{\mbox{\rm\tiny +}}E=h^*E/K_{-1}$ with quotient map
$\pi:h^*E\to \mathscr{G}$. Define $\sigma=\pi\circ h^*s:h^*G\to\mathscr{G}$.
Since $i^*K_{-1}=\Lie(G,I)$ as a subgroup of $E$, we see that
$i^*\mathscr{G}\simeq G$ and $i^*\sigma$ is the identity of $G$.
Since $G$ is $k$-flat, it follows that $\sigma$ is an
isomorphism. From the construction of $K_{-1}$, we see that
after we compose with the isomorphisms
\[
h^*\tau_s:h^*h_*h^*G\isomto h^*E \quad\mbox{and}\quad
\sigma:h^*G\isomto \mathscr{G},
\]
the flat surjection $\pi:h^*E\to \mathscr{G}$ is identified with the
counit of the adjunction:
\[
\beta_{h^*G}:h^*h_*h^*G\too h^*G.
\]
We apply $h_*$ and obtain the commutative diagram:
\[
\xymatrix@R=10mm@C=15mm{
h_*h^*G \ar[r]^-{\alpha_{h_*h^*G}}
\ar[d]^{\mbox{\rotatebox{90}{$\sim$}}} \ar[d]_{\tau_s}
& h_*h^*h_*h^*G \ar[r]^-{h_*\beta_{h^*G}} 
\ar[d]^{\mbox{\rotatebox{90}{$\sim$}}} \ar[d]_{h_*h^*\tau_s}
& h_*h^*G \ar[d]^{\mbox{\rotatebox{90}{$\sim$}}}
\ar[d]_{h_*\sigma} \\
E \ar[r]^-{\alpha_E} & h_*h^*E \ar[r]^{h_*\pi} & h_*\mathscr{G}}
\]
Since the top row is the identity, we see that the bottom
row is an isomorphism, i.e. $E\isomto h_*\mathscr{G}$. Again,
it is clear that this isomorphism is functorial.

Finally, we consider the commutation with base changes.
For Weil restriction, this is a standard fact. For Weil
extension, this follows from base change commutation for
pullbacks and for quotients by flat subgroups.
\end{item-title}

\begin{item-title}{Proof of~\ref{main_theorem}(2)}
Let $1\to \mathscr{G}'\to\mathscr{G}\to\mathscr{G}''\to 1$
be an exact sequence in $\Gr\!/k[I]$. Then the exact
sequences with solid arrows are exact:
\[
\renewcommand{\arraystretch}{1.5}
\begin{array}{c}
1\too \mathscr{G}'_k\too\mathscr{G}_k\too\mathscr{G}''_k\too 1 \\
\xymatrix@C=6mm{
1 \ar[r] & \Lie(\mathscr{G}'_k,I) \ar[r] & \Lie(\mathscr{G}_k,I) \ar[r] &
\Lie(\mathscr{G}''_k,I) \ar@{..>}[r] & 1}
\end{array}
\]
Moreover, if $\mathscr{G}'$ is smooth then the second sequence is
exact also if we include the dotted arrow. By an easy diagram
chase, we find that the commutative diagram below has exact
rows (again including dotted arrows if $\mathscr{G}'$ is smooth):
\[
\xymatrix@C=6mm@R=6mm{
& 1 \ar[d] & 1 \ar[d] & 1 \ar[d] & \\
1 \ar[r] & \Lie(\mathscr{G}'_k,I) \ar[r] \ar[d]
& \Lie(\mathscr{G}_k,I) \ar[r] \ar[d]
& \Lie(\mathscr{G}''_k,I) \ar@{..>}[r] \ar[d] & 1 \\
1 \ar[r]& h_*\mathscr{G}'_k \ar[r] \ar[d] & h_*\mathscr{G}_k \ar[r] \ar[d]
& h_*\mathscr{G}''_k \ar@{..>}[r] \ar[d] & 1 \\
1 \ar[r] & \mathscr{G}'_k \ar[r] \ar[d] & \mathscr{G}_k \ar[r] \ar[d]
& \mathscr{G}''_k \ar[r] \ar[d] & 1 \\
& 1 & 1 & 1 & }
\]
This proves the claim.
\end{item-title}

\begin{item-title}{Proof of~\ref{main_theorem}(4)}
If $\mathscr{G}$ is of finite type, or smooth, or connected, or unipotent,
or split unipotent, or solvable, then $G=i^*\mathscr{G}$ as well as
$\Lie(G,I)$ have the same property. It follows that $E=h_*\mathscr{G}$
has the property. Moreover, if $\mathscr{G}$ is commutative then $E$ also.

Conversely, if $E$ is of finite type, or smooth, or connected, or unipotent, or split unipotent, or solvable, or commutative, then
$h^*E$ has the same property. Therefore the quotient
$h^{\mbox{\tiny +}}E:=h^*E/K_{-1}$ has the same property.
\end{item-title}

\subsection{Proof of the main theorem:
isomorphism of $\OO_k$-module stacks}
\label{subsection:iso of module stacks}

In this paragraph, we prove~\ref{main_theorem}(3), i.e.
that the Weil restriction functor
$h_*:\Gr\!/k[I]\to\Ext(I)/k$ exchanges the addition and the
scalar multiplication on both sides. Before we start, we point
out that these properties will imply that the image of a trivial
deformation group scheme $\mathscr{G}=h^*G$ under Weil restriction is
the tangent bundle (i.e. trivial) extension $\T(G,I)$,
a fact which can be shown directly using
Proposition~\ref{ker of beta}(4).

We work in the fibre
category over a fixed $G\in\Gr\!/k$ and we set $L:=\Lie(G,I)$.
Let $\mathscr{G}_1,\mathscr{G}_2\in\Gr\!/k[I]$
with identifications $i^*\mathscr{G}_1\simeq G\simeq i^*\mathscr{G}_2$.
For clarity, we introduce three copies $I_1=I_2=I$ of the
same finite free $k$-module. For $c=1,2$ we have obvious maps:
\[
\xymatrix@C=5mm{
\Spec(k[I_c]) \ar[rr]^{j_c} \ar[rd]_{h_c}
& & \Spec(k[I_1\oplus I_2]) \ar[ld]^{\ell} \\
& \Spec(k) & }
\]
Set $\mathscr{G}'=\mathscr{G}_1\amalg_G\mathscr{G}_2 \in \Gr\!/k[I_1\oplus I_2]$,
so $\mathscr{G}_1+\mathscr{G}_2=j^*\mathscr{G}'$ where
$j:\Spec(k[I])\into \Spec(k[I\oplus I])$ is the closed
immersion induced by the addition morphism $I\oplus I\to I$.
We have a morphism
\[
\xi:\ell_*\mathscr{G}'\too h_{1,*}\mathscr{G}_1\times_Gh_{2,*}\mathscr{G}_2
\]
whose component $\xi_c:\ell_*\mathscr{G}'\to h_{c,*}\mathscr{G}_c$ is
the $\ell_*$-pushforward of the adjunction
$\mathscr{G}'\to j_{c,*}j_c^*\mathscr{G}'=j_{c,*}\mathscr{G}_c$. Besides, we have
a morphism
\[
\omega:\ell_*\mathscr{G}'\too h_*(\mathscr{G}_1+\mathscr{G}_2)
\]
obtained as the $\ell_*$-pushforward of the adjunction
$\mathscr{G}'\to j_*j^*\mathscr{G}'=j_*(\mathscr{G}_1+\mathscr{G}_2)$.

\begin{lemma} \label{lemma:compatibility with addition}
The morphism $\xi$ is an isomorphism and it induces an
isomorphism of extensions on the bottom row of the
following commutative square:
\[
\xymatrix@C=16mm{
\ell_*\mathscr{G}' \ar@{->>}[d]_-{\omega} \ar[r]^(.45){\overset{\xi}{\sim}} &
h_{1,*}\mathscr{G}_1\times_Gh_{2,*}\mathscr{G}_2 \ar@{->>}[d] \\
h_*(\mathscr{G}_1+\mathscr{G}_2) \ar[r]^(.45){\sim} & h_{1,*}\mathscr{G}_1+h_{2,*}\mathscr{G}_2}
\]
\end{lemma}

\begin{proof}
Write $G=\Spec(A)$ and $\mathscr{G}_c=\Spec(\mathscr{A}_c)$ so
$\mathscr{G}'=\Spec(\mathscr{A}_1\times_A\mathscr{A}_2)$. There is a morphism of algebra
schemes $\xi':\ell_*\OO[\mathscr{G}']\to
h_{1,*}\OO[\mathscr{G}_1]\times_{\OO_k[G]}h_{2,*}\OO[\mathscr{G}_2]$
contructed in the same way as $\xi$. It order to describe $\xi'$
we can express the Weil restrictions in terms of $I$-compatible maps
as in Lemma~\ref{lemma:description of Weil res isom}. For a
$k$-algebra~$R$, we have:
\[
\setlength{\arraycolsep}{.5mm}
\renewcommand{\arraystretch}{.5}
\begin{array}{rcl}
\Homc_k(\mathscr{A}_1\times_A\mathscr{A}_2,I_1R\oplus I_2R)
& \stackrel{\xi'(R)}{\tooo} &
\Homc_k(\mathscr{A}_1,I_1R)\underset{\Hom_k(A,R)}{\times}\Homc_k(\mathscr{A}_2,I_2R) \\
v & \longmapsto & (v_1,v_2)
\end{array}
\]
where $v_1$ (resp. $v_2$) is $v$ modulo $I_2$ (resp. $I_1$).
This is a bijection whose inverse sends a pair $(v_1,v_2)$ with
$v_1^*=v_2^*:A\to R$ to the map
$v:\mathscr{A}_1\times_A\mathscr{A}_2\to I_1R\oplus I_2R$,
$(a_1,a_2)\mapsto v_1(a_1)+v_2(a_2)$. The morphism $\xi$
is the bijection obtained by restriction of $\xi'$ to the subsets
of algebra maps as in
Lemma~\ref{lemma:description of Weil res isom}(2). Namely,
an algebra map is of the form $f={\bar v}+v$ where $v$ is $I$-compatible,
and $\omega(R)$ sends $f$ to $(f_1,f_2)$ while $\omega(R)^{-1}$
sends $(f_1={\bar v}+v_1,f_2={\bar v}+v_2)$ to $f={\bar v}+v_1+v_2$.

In order to describe $\omega$ note that
$(h_*(\mathscr{G}_1+\mathscr{G}_2))(R)=
\Hom_{k[I_1\oplus I_2]\mbox{-}\Alg}(\mathscr{A}_1\times_A\mathscr{A}_2,R[I])$
where $R[I]$ is a $k[I_1\oplus I_2]$-algebra via the
map $k[I_1\oplus I_2]\to k[I]$ induced by addition
$+:I\oplus I\to I$. Then $\omega(R)$ sends $f$ to the composition
\[
\mathscr{A}_1\times_A\mathscr{A}_2 \stackrel{f}{\too} R[I_1\oplus I_2]
\stackrel{+}{\too} R[I].
\]
Thus $\omega(R)$ is surjective because $R[I_1\oplus I_2]\too R[I]$
has $R$-algebra sections, i.e. $\omega$ is a surjection of functors.
Its kernel is the set of maps $f={\bar v}+v_1+v_2$ such that
$v_1+v_2=d_1+d_2:\mathscr{A}_1\times_A\mathscr{A}_2\to R[I]$, with
$e_c=d_c^*+d_c:\mathscr{A}_c\to k[I_c]$ the counits of the Hopf algebras.
After translation by the derivations as indicated by
Proposition~\ref{ker of beta}(3), on the
side of extensions the kernel is $\ker(+:L\times L\to L)$,
giving rise to a quotient isomorphic to the Baer sum
extension $h_{1,*}\mathscr{G}_1+h_{2,*}\mathscr{G}_2$.
\hfill $\Box$
\end{proof}

It remains to prove that $h_*:\Gr\!/k[I]\to\Ext(I)/k$ exchanges
the scalar multiplication on both sides. Let $\mathscr{G}\in\Gr\!/k[I]$
with an identification $i^*\mathscr{G}\simeq G$. We will reduce to a
situation similar as that of
Lemma~\ref{lemma:compatibility with addition} thanks to the
following trick.

\begin{lemma}
Let $j_\lambda:\Spec(k[I])\into \Spec(k[I_1\oplus I_2])$
be the closed immersion defined by the surjective $k$-algebra
map $I_1\oplus I_2\to I$, $i_1\oplus i_2\mapsto \lambda i_1+i_2$.
Then we have
$s_\lambda^*\mathscr{G}\isomto j_\lambda^*(\mathscr{G}\amalg_G h^*G)$
canonically.
\end{lemma}

\begin{proof}
If we think of $\Spec k[I_1\oplus I_2]$ as the coproduct
$\Spec k[I_1]\amalg_{\Spec k} \Spec k[I_2]$, the map
$j_\lambda$ is the composition $(s_\lambda\amalg \id)\circ j$
as follows:
\[
\xymatrix@C=12mm{
\Spec(k[I]) \ar[r]^-j & \Spec(k[I_1\oplus I_2])
\ar[r]^-{s_\lambda\amalg \id} &
\Spec(k[I_1\oplus I_2]).}
\]
It follows that:
\[
j_\lambda^*(\mathscr{G}\amalg_G h^*G)
=j^*((s_\lambda\amalg \id)^*(\mathscr{G}\amalg_G h^*G))
=j^*(s_\lambda^*\mathscr{G}\amalg h^*G)
=s_\lambda^*\mathscr{G}+h^*G
=s_\lambda^*\mathscr{G},
\]
because $h^*G$ is the neutral element for the sum in the
fibre category of $\Gr\!/k[I]\to\Gr\!/k$ at $G$.
\hfill $\Box$
\end{proof}

Set $\mathscr{G}_1=\mathscr{G}$ and $\mathscr{G}_2=h^*G$. Recall that $L=\Lie(G,I)$.
The Weil restrictions are $E:=E_1=h_*\mathscr{G}$ and the trivial
extension $E_2=h_*h^*G=L\rtimes G$. As in
Lemma~\ref{lemma:compatibility with addition}, there are
morphisms $\xi:\ell_*\mathscr{G}'\to E_1\times_G E_2=E\times L$.

\begin{lemma} \label{lemma:compatibility with scal mult}
The morphism $\xi$ is an isomorphism which induces an
isomorphism of extensions on the bottom row of the following
commutative square:
\[
\xymatrix@C=16mm{
\ell_*\mathscr{G}' \ar@{->>}[d]_-{\omega}
\ar[r]^(.45){\overset{\xi}{\sim}} & h_*\mathscr{G}\times L \ar@{->>}[d] \\
h_*(\lambda\mathscr{G}) \ar[r]^(.45){\sim} & \lambda h_*\mathscr{G}.}
\]
\end{lemma}

\begin{proof}
The proof is the same as that of
Lemma~\ref{lemma:compatibility with addition} except that
in the final step we use the map $I\oplus I\to I$,
$i_1\oplus i_2\mapsto \lambda i_1+i_2$. Again this morphism
is surjective and on the side of extensions, the kernel
corresponds to the kernel of $L\times L\to L$,
$(v_1,v_2)\mapsto \lambda v_1+v_2$. The quotient of $E\times L$
by this kernel is exactly the extension $\lambda E$, the
pushout of the diagram:
\[
\xymatrix{
L \ar[r] \ar[d]_-{\lambda}
\ar@{}[rd]|{\mbox{\LARGE$\ulcorner$}}
& E \ar[d] \\
L \ar[r] & \lambda E .}
\]
This finishes the proof.
\hfill $\Box$
\end{proof}

\section{Dieudonn\'e theory for unipotent groups
over the dual numbers} \label{section:application}

In this section, as an application of Theorem~\ref{main_theorem_0},
we give a classification of smooth, unipotent group schemes
over the dual numbers of a perfect field $k$, in terms
of extensions of Dieudonn\'e modules. So throughout the
section, the ring $k$ is a perfect field of characteristic
$p>0$.

\subsection{Reminder on Dieudonn\'e theory}

We denote by $W$ the Witt ring scheme over $k$ and $\Fr,\V$
its Frobenius and Verschiebung endomorphisms.
For all $n\ge 1$, we write $W_n:=W/V^n W$ the ring scheme
of Witt vectors of length $n$. We use the same notation also
for these operators over the $R$-points, with $R$ a $k$-algebra.
We also define $\tilde{\V}:W_n\too W_{n+1}$ as the morphism
induced on $W_n$ by the composition 
\[
\xymatrix{
W \ar[r]^-{\V} & W \ar[r]^-{\pi_{n+1}} & W_{n+1}
}
\]
where $\pi_{n+1}$ is the natural projection.

The {\em Dieudonn\'e ring $\DD$} is the $W(k)$-algebra
generated by two variables $\FF$ and $\VV$ with the relations:
\[
\setlength{\arraycolsep}{.7mm}
\renewcommand{\arraystretch}{1.2}
\begin{array}{rl}
\FF x & =\Fr(x)\FF\\
x\VV  & =\VV \Fr(x) \\
\FF \VV & =\VV\FF =p,
\end{array}
\]
for varying $x\in W(k)$.
A {\em Dieudonn\'e module} is a left $\DD$-module.
A Dieudonn\'e module $M$ is called {\em erasable} if for any
$m\in M$ there exists a positive integer $n$ such that $\VV^nm=0$.

Let $R$ be a $k$-algebra. Then, for any $n\ge 1$, we make
$W_n(R)$ a left $\DD$-module with the rules:
\[
\setlength{\arraycolsep}{.7mm}
\renewcommand{\arraystretch}{1.2}
\begin{array}{rl}
\FF \cdot u & =\Fr(u) \\
\VV \cdot u & =\V(u) \\
x \cdot u & = \Fr^{1-n}(x)u
\end{array}
\]
for all $u\in W_n(R)$ and $x\in W(k)$.
The twist in the latter definition is designed to make
$\tilde{V}:W_n(R)\to W_{n+1}(R)$ a morphism of $\DD$-modules,
see Demazure and Gabriel ~\cite[chap.~V, \S~1, no~3.3]{DG70}.
All of this is functorial
in $R$ and gives $W_n$ a structure of $\DD$-module scheme.
In particular, $\End_k(W_n)$ is a $\DD$-module. According
to~\cite[chap.~V, \S~1, no~3.4]{DG70} the morphism
$\DD \to \End_k(W_n)$ induces an isomorphism of $\DD$-modules:
\[
\DD/\DD V^n \isomto \End_k(W_n).
\]
If $U$ is a commutative, unipotent $k$-group scheme, the set
$\Hom_k(U,W_n)$ is a Dieudonn\'e module with its structure
given by postcomposition, i.e. for any $f:U\to W_n$:
\[
\setlength{\arraycolsep}{.7mm}
\renewcommand{\arraystretch}{1.2}
\begin{array}{rl}
\FF\cdot f &=\Fr \circ f \\
\VV\cdot f &=\V\circ f \\
x\cdot f &= \Fr^{1-n}(x)f,
\end{array}
\]
all $x\in W(k)$.
We define the {\em Dieudonn\'e module of $U$} as:
\[
\underline M(U):=\lim_{\stackrel{\too}{n}} \Hom_k(U,W_n)
\]
where the transition maps of the inductive system are
induced by $\tilde{V}:W_n\to W_{n+1}$. Since $\Hom_k(U,W_n)$
is killed by $\VV^n$ and $\underline M(U)$ is a union of these
subgroups, we see that $\underline M(U)$ is erasable. If $M$
is a Dieudonn\'e
module, we define its {\em Frobenius twist $M^{(p)}$} as the
module with underlying group $M^{(p)}=M$ and $\DD$-module
structure given by:
\[
\FF_{M^{(p)}}=\FF_M, \quad \VV_{M^{(p)}}=\VV_M,
\quad x_{M^{(p)}}=\Fr^{-1}(x)_M \quad \mbox{for all } x\in W(k).
\]
Then the maps $\FF_M:M^{(p)}\to M$ and $\VV_M:M\to M^{(p)}$
are visibly $\DD$-linear. Moreover, let $\CU/k$ be the category
of commutative unipotent $k$-group schemes; according to
\cite[chap.~V, \S~1, 4.5]{DG70} we have a canonical isomorphism
$\underline M(U)^{(p)}\isomto \underline M(U^{(p)})$ for all
$U\in \CU$ with Frobenius twist $U^{(p)}$.

\begin{theorem} \label{theo: Dieudonne modules}
Let $\DD\mbox{-}\Mod$ be the category of Dieudonn\'e
modules. The contravariant functor 
$$
\underline M:\CU/k \too \DD\mbox{-}\Mod
$$
is exact, fully faithful with essential image the full
subcategory of erasable Dieudonn\'e modules. It
transforms the Frobenius (resp. Verschiebung) of $U$ into
the Frobenius (resp. Verschiebung) of $\underline M(U)$.
Moreover a unipotent group scheme $U$ is of finite type
if and only if $\underline M(U)$ is of finite type.
\end{theorem}

\begin{proof}
See~\cite[chap.~V, \S~1, 4.3, 4.4, 4.5]{DG70}.
\hfill $\Box$
\end{proof}

\subsection{Dieudonn\'e theory over the dual numbers}

Before stating our Dieudonn\'e classification, we need to
define the notions of {\em Lie algebra} and {\em smoothness}
of Dieudonn\'e modules. We let
${\DD\mbox{-}\Mod}^{e}\subset {\DD\mbox{-}\Mod}$ be the
subcategory of erasable $\DD$-modules.

\begin{definition}
Let $M\in \DD\mbox{-}\Mod$ be a Dieudonn\'e module.
We define the {\em Lie algebra of $M$} by:
\[
\Lie M:=(M/\FF M)\otimes_k k[\FF].
\]
If $I$ a finite dimensional $k$-vector space,
the {\em $I$-Lie algebra of $M$} is
$\Lie(M,I):=\Lie M\otimes_k I^\smallvee$.
This gives rise to endofunctors $\Lie(-)$ and $\Lie(-,I)$
of the category ${\DD\mbox{-}\Mod}^{e}$.
\end{definition}

\begin{prop} \label{prop:M commutes with Lie}
We have an isomorphism, functorial in $U\in \CU/k$:
\[
\Lie(\underline M(U),I) \isomto \underline M(\Lie(U,I)).
\]
\end{prop}

\begin{proof}
We start with the case of dimension one $I=k$, so
$\Lie(U,k\eps)=\Lie(U)$. Let $U'$ be the kernel of Frobenius
in $U$ and $M=\underline M(U)$, $M'=\underline M(U')$.
We have an exact sequence:
\[
0 \too U'\too U\stackrel{\Fr_U}{\too} U^{(p)}.
\]
We deduce isomorphisms $\Lie U\isomto \Lie U'$ and
$M'=M/\FF M$. In the sequel set $L_U:=(\Lie U)(k)$,
a $k$-vector space. Since $U'$ is a finite commutative group scheme, according to Fontaine~\cite[chap.~III, 4.2]{Fo77} there is a canonical isomorphism:
\[
\eta_{U'}:M'/\FF M' \isomto L_{U'}^\smallvee.
\]
We deduce a composed isomorphism $\eta_U$ as follows:
\[
M/\FF M \isomto M'/\FF M' \stackrel{\eta_{U'}}{\isomto}
L_{U'}^\smallvee \isomto L_U^\smallvee.
\]
From the isomorphism of $k$-group schemes
$\Lie U\simeq \VV(L_U^\smallvee)$, we deduce
$$
M(\Lie U)=\Hom_k(\VV(L_U^\smallvee),\GG_a)\isomto L_U^\smallvee\otimes k[\FF].
$$
By tensoring $\eta_U$ with $k[\FF]$, we find
\[
\Lie M=(M/\FF M)\otimes_k k[\FF] \isomto
L_U^\smallvee\otimes k[\FF] \isomto M(\Lie U).
\]
The result for general $I$ follows since
$\Lie(U,I)\simeq \Lie U\otimes_{\OO_k}\VV(I^\smallvee)$
and
$\Lie(M,I)=\Lie M\otimes_k I^\smallvee$.
\hfill $\Box$
\end{proof}

We can characterize the functor $\Lie$
on Dieudonn\'e modules by its values on the modules
$\DD/\DD\VV^n$.

\begin{prop} \label{prop:Lie algebra of D-modules}
There exists a unique covariant functor
$\mathscr{L}:{\DD\mbox{-}\Mod}^{e}\to {\DD\mbox{-}\Mod}^{e}$
with the following properties:
\begin{trivlist}
\itemn{1} $\mathscr{L}$ is right exact and commutes with filtering
inductive limits;
\itemn{2} $\mathscr{L}(\DD/\DD\VV^n)=k[\FF]^n$ for all $n\ge 1$;
\itemn{3} $\mathscr{L}:\End(\DD/\DD\VV^n) \to\End(k[\FF]^n)$ sends
\begin{itemize}
\item $\FF$ to $0$;
\item $\VV$ to the endomorphism
$(a_0,a_1,\dots,a_{n-1})\mapsto(0,a_0,a_1,\dots,a_{n-2})$;
\item and multiplication by $x=(x_0,x_1,x_2,\dots)\in W(k)$
to the diagonal endomorphism with diagonal entries
$(x_0,x_0^p,\dots,x_0^{p^{n-1}})$.
\end{itemize}
\end{trivlist}
\end{prop}

\begin{proof}
Uniqueness. The key
is the fact that $\DD/\DD\VV^n$ is a projective generator of
the full subcategory $C_n:=({\DD\mbox{-}\Mod}^{e})^{\VV^n=0}$
of objects killed by $\VV^n$. More precisely, since any
erasable $\DD$-module is a filtering union of its submodules
of finite type, property (1) implies that $\mathscr{L}$ is determined
by its restriction to the subcategory of modules of finite type.
Any finite type module $M$ is killed by $\VV^n$ for some
$n\ge 1$. Since $\DD/\DD\VV^n$ is noetherian, for any
$M\in C_n$ there exist $r,s$ and an exact sequence:
\[
(\DD/\DD\VV^n)^s \too (\DD/\DD\VV^n)^r\too M\too 0,
\]
and for any morphism $f:M\to M'$ in $C_n$ there is a
commutative diagram:
\[
\xymatrix{
(\DD/\DD\VV^n)^s \ar[r] \ar[d]^{h} &
(\DD/\DD\VV^n)^r \ar[r] \ar[d]^{g} & M \ar[r] \ar[d]^f & 0 \\
(\DD/\DD\VV^n)^{s'} \ar[r] &
(\DD/\DD\VV^n)^{r'} \ar[r] & M' \ar[r] & 0.}
\]
Since $\mathscr{L}$ is right exact, this implies that $\mathscr{L}(M)$
is determined by the values of $\mathscr{L}((\DD/\DD\VV^n)^r)$ for
variable~$r$, and $\mathscr{L}(f)$ is determined by the values
of $\mathscr{L}(g)$ for variable $g$ as above. Again since $\mathscr{L}$
is right exact, it is additive. Hence $\mathscr{L}((\DD/\DD\VV^n)^r)$
is determined by $\mathscr{L}(\DD/\DD\VV^n)$ which is prescribed in (2).
Similarly $\mathscr{L}(g)$ is determined by the values of $\mathscr{L}$ on the
various maps
$\DD/\DD\VV^n\into (\DD/\DD\VV^n)^r \twoheadrightarrow
\DD/\DD\VV^n$. Since $\End(\DD/\DD\VV^n)=\DD/\DD\VV^n$
is generated by $W(k)$, $F$ and $V$, the assignment in (3)
ensures uniqueness of $\mathscr{L}$.

\smallskip

\noindent Existence. Because of
Proposition~\ref{prop:M commutes with Lie}, it is enough to
check that the functors $\mathscr{L}(M(-))$ and $M(\Lie(-))$ take
the same values on the Witt groups $U=W_n$ and the
endomorphisms of these groups. This is a simple computation
which is left to the reader.
\hfill $\Box$
\end{proof}

We come to the notion of smoothness.
It is known that a $k$-group scheme of finite type $U$ is
smooth if and only if its relative Frobenius
$\Fr_{U/k}:U\to U^{(p)}$ is an epimorphism of $k$-group schemes.
This motivates the following definition.

\begin{definition}
An erasable Dieudonn\'e module $M$ is called {\em smooth}
if it is of finite type and its Frobenius morphism
$\Fr_M:M^{(p)}\to M$ is a monomorphism.
\end{definition}

\begin{definition}
An {\em $I$-extension} in ${\DD\mbox{-}\Mod}^{e}$ is
an extension of smooth erasable Dieudonn\'e modules of
the form
\[
0 \too M \too M'\too \Lie(M,I) \too 0.
\]
A {\em morphism of $I$-extensions} is a morphism of
extensions in the usual sense, that is, a morphism of
short exact sequences of $\DD$-modules.
\end{definition}

\begin{theorem} \label{theorem:Dieudonne for SCU}
Let $\SCU/k[I]$ be the category of smooth, commutative,
unipotent (i.e. with unipotent special fibre)
$k[I]$-group schemes. Let
$\DD\mbox{-}I\mbox{-}\Mod$ be the category of 
$I$-extensions of smooth erasable Dieudonn\'e modules.
Then the Dieudonn\'e functor $M$ induces a contravariant
equivalence of categories:
\[
\mathscr{M}:\SCU/k[I] \too \DD\mbox{-}I\mbox{-}\Mod
\]
that sends $\mathscr{U}$ to the extension
$0 \to \underline M(\mathscr{U}_k)\to \underline M(h_*\mathscr{U}))
\to \underline M(\Lie(\mathscr{U}_k,I))\to 0$.
A quasi-inverse functor is obtained by sending an extension
$0 \to M \to M'\to \Lie(M,I) \to 0$ to the Weil extension
$h^{\mbox{\rm\tiny +}}(\underline U(M'))$ of the extension
$0 \to \Lie(\underline U(M),I)\to \underline U(M')\to
\underline U(M)\to 0$, where $\underline U$ is a quasi-inverse
for $\underline M$.
\end{theorem}

\begin{proof}
It suffices to put together Theorem~\ref{main_theorem_0}
and Theorem~\ref{theo: Dieudonne modules}. In little more
detail, let $\mathscr{U}$ be a smooth, commutative, unipotent group
scheme over the ring of dual numbers $k[I]$, and let $U=\mathscr{U}_k$
be its special fibre. By Theorem~\ref{main_theorem_0}
this datum is equivalent to an extension
$$
0\too \Lie(U,I)\too E\too U\too 0
$$
with $E=h_*\mathscr{U}$ smooth, commutative, unipotent. By
Theorem~\ref{theo: Dieudonne modules} this is equivalent
to an extension
$$
0\too \underline M(U)\too \underline M(E)
\too \underline M(\Lie(U,I))\too 0.
$$
Since
$\Lie(\underline M(U),I) \isomto \underline M(\Lie(U,I))$
by Proposition~\ref{prop:M commutes with Lie},
we obtain an $I$-extension as desired.
\hfill $\Box$
\end{proof}

\appendix{A}{Differential calculus on group schemes}
\label{section:differential calculus}

In this appendix
we review the notions of tangent bundle and Lie algebra in the
required generality. We introduce the exponential morphism
of a $k$-group scheme and we establish its main properties,
including some special properties needed in the paper. Finally
we show how the use of the group algebra allows to recover
easily the deformation theory of smooth affine group schemes.
Proofs are often omitted, especially in~\ref{TG and Lie G}
and~\ref{subsection:exp and inf translation}.

\subsection{Tangent bundle and Lie algebra}
\label{TG and Lie G}

Let $k$ be a ring and $I$ a free $k$-module of finite rank $r\ge 1$
with dual $I^\smallvee=\Hom_k(I,k)$.
Let $k[I]$ be the algebra of dual numbers, i.e. $k[I]:=k\oplus I$
with multiplication determined by the condition $I^2=0$.
Let $h:\Spec(k[I])\to\Spec(k)$ be the structure map and
$i:\Spec(k)\to\Spec(k[I])$ the closed immersion. Basic
structure facts on the ring schemes $\OO_k$ and $\OO_{k[I]}$
are recalled in Paragraph~\ref{defs:base ring scheme}.

Let $G$ be a $k$-group scheme with unit section
$e:\Spec(k)\to G$. The {\em tangent bundle of $G/k$ relative
to~$I$} is the $k$-group scheme defined by:
\[
\T(G,I):=h_*h^*G.
\]
The $(h^*,h_*)$ adjunction (see
Subsection~\ref{subsection:Weil restriction}) provides
two morphisms of group schemes:
\[
\alpha_G:G\too \T(G,I) \quad , \quad
\beta_{h^*G}:h^*\T(G,I)=h^*h_*h^*G\too h^* G.
\]
From these we derive $\pi_G:=i^*\beta_{h^*G}:\T(G,I)\to G$
and the {\em Lie algebra of $G/k$ relative to $I$}:
\[
\Lie(G,I):= \ker(\pi_G).
\]
The map $\T(G,I)\to G$ is a vector bundle which can be
described in terms of derivations. If we identify the
$k$-module $I$ with the corresponding locally free sheaf
on $\Spec(k)$, then for all points $f:S\to G$ with
values in a $k$-scheme $S$ we have
(\cite[Expos\'e~II, Prop.~3.3]{SGA3.1}):
\[
\Hom_G(S,\T(G,I)) \isomto \Der^f(\cO_G,f_*\cO_S\otimes_k I).
\]
In particular $\Lie(G,I)\to\Spec(k)$ is an $\OO_k$-Lie algebra
scheme such that for all $S/k$ we have
\[
\Hom_k(S,\Lie(G,I))\isomto \Der^e(\cO_G,e_*\cO_S\otimes_kI).
\]
It supports the {\em adjoint representation}, i.e. the action
of $G$ by conjugation inside $\T(G,I)$:
\[
\Ad:G\to\GL(\Lie(G,I)) \quad,\quad
\Ad(g)(x)=\alpha_G(g)x\alpha_G(g)^{-1}
\]
for all points $g$ of $G$ and $x$ of $\Lie(G,I)$.
Applying the functor Lie, that is differentiating
at the unit section of $G$, we obtain the infinitesimal
adjoint representation of the Lie algebra:
\[
\ad:\Lie G\to\End(\Lie(G,I))).
\]
When $I=k\eps$, from $\ad$ we deduce the bilinear form
called {\em bracket} $[-,-]:\Lie G\times \Lie G\to\Lie G$.
That is, we have
$[x,y]:=(\ad x)(y)$ for all points $x,y$ of $\Lie G$.

The tangent bundle $\T(G,I)$ carries a structure of extension
as follows. By the triangle identity of adjunction,
$h^*\alpha_G$ is a section of $\beta_{h^*G}$, hence $\alpha_G$
is a section of $\pi_G$. Letting $\gamma_G$
be the inclusion of $\Lie(G,I)$ into $\T(G,I)$,
we have a split exact sequence:
\[
\xymatrix@C=8mm{
1 \ar[r] & \Lie(G,I) \ar[r]^-{\gamma_G}
& \T(G,I) \ar@<.5ex>[r]^-{\pi_G} & G \ar[r]
\ar@<.5ex>[l]^-{\alpha_G} & 1.}
\]
This is an exact sequence of functors, hence
also an exact sequence of group schemes (i.e. of fppf
sheaves). Let $m$ be the multiplication of $\T(G,I)$.
The splitting gives rise to an isomorphism of $k$-schemes:
\[
\xymatrix@C=12mm{
\varrho_G:\Lie(G,I)\times G
\ar[r]^-{\gamma_G\times\alpha_G} &
\T(G,I)\times\T(G,I) \ar[r]^-m & \T(G,I),}
\]
that is, any point of $\T(G,I)$ may be written uniquely as
a product $\gamma_G(x)\cdot\alpha_G(g)$ for some points
$x\in\Lie(G,I)$ and $g\in G$. We will sometimes
write briefly $(x,g)=\gamma_G(x)\cdot\alpha_G(g)$ to denote
this point of $\T(G,I)$. The conjugation action of $G$ on
$\Lie(G,I)$ related
to the extension structure is given by the adjoint action,
thus the group structure of $\T(G,I)$ can be described by:
\[
(x,g)\cdot (x',g')=(x+\Ad(g)x',gg').
\]

The dependence of $\T(G,I)$ and $\Lie(G,I)$ on $I$ can be
described further. If $I=k\eps$ so that $k[I]=k[\eps]$ with
$\eps^2=0$, we write simply $\T\!G=\T(G,k\eps)$ and
$\Lie G=\Lie(G,k\eps)$ and we call them the {\em tangent bundle}
and the {\em Lie algebra} of $G$. Let $u:G\to \Spec(k)$ be the structure map. For general $I$, the isomorphisms
$\Der^f(\cO_G,f_*\cO_S)\otimes_kI\isomto
\Der^f(\cO_G,f_*\cO_S\otimes_k I)$
functorial in $S/G$ induce an isomorphism of vector bundles:
\[
\T\!G\otimes_{\,\OO_G}\!u^*\VV(I^\smallvee)\isomto \T(G,I).
\]
Similarly, the isomorphisms $\Der(\cO_G,f_*\cO_S)\otimes_kI\isomto
\Der(\cO_G,f_*\cO_S\otimes_kI)$
functorial in $S/k$ induce an isomorphism of $\OO_k$-Lie
algebra schemes:
\[
\Lie G\otimes\VV(I^\smallvee)\isomto \Lie(G,I).
\]
With $G$ acting trivially on $\VV(I^\smallvee)$, this
isomorphism is $G$-equivariant.

\begin{item-title}{Idempotence of Lie}
An important property for us will be the idempotence
of the Lie functor, namely the existence of an isomorphism
$d:\Lie(G,I) \to \Lie(\Lie(G,I))$ as in
\cite[Expos\'e~II, 4.3.2]{SGA3.1}. To describe it, let $I,J$
be two finite free $k$-modules. Note that
$R[I][J]=R\oplus IR\oplus JR\oplus IJR$. If $G=\Spec(A)$
is affine, for a $k$-algebra $R$, the elements of the set
$G(R[I][J])$ are the maps $f:A\to R[I][J]$ written
$f(x)=r(x)+s(x)+t(x)+u(x)$ where $r:A\to R$ is an algebra map,
$s:A\to IR$ and $t:A\to JR$ are $r$-derivations, and
$u:A\to IJR$ satisfies the identity:
\[
u(xy)=r(x)u(y)+r(y)u(x)+s(x)t(y)+s(y)t(x).
\]
Thus $u$ is an $r$-derivation if $s=0$ or $t=0$.
Consider the two maps
$p:R[I][J]\to R[I]$, $J\mapsto 0$ and
$q:R[I][J]\to R[J]$, $I\mapsto 0$.
Unwinding the definition we see that:
\[
(\Lie(\Lie(G,I),J)(R)
=\ker\big(G(R[I][J])
\stackrel{(p,q)}{\tooo} G(R[I])\times G(R[J])\big).
\]
For varying $R$, the maps $R[I\otimes_k J]\to R[I][J]$,
$i\otimes j\mapsto ij$ induce a morphism of Lie algebra schemes
$d:\Lie(G,I\otimes_kJ) \too \Lie(\Lie(G,I),J)$.

\begin{lemma} \label{lemma:isom d}
The morphism $d:\Lie(G,I\otimes_kJ) \to \Lie(\Lie(G,I),J)$
is an isomorphism. \hfill $\square$
\end{lemma}

\bigskip

For simplicity of notation, we will write $d$ as an
equality:
$\Lie(G,I\otimes_kJ)=\Lie(\Lie(G,I),J)$. This will
not cause any ambiguity. If $I=J=k\eps$ this means simply
that $\Lie G=\Lie(\Lie G)$.
\end{item-title}

\subsection{Exponential and infinitesimal translation}
\label{subsection:exp and inf translation}
Demazure and Gabriel in~\cite{DG70} use an exponential notation
which is flexible enough to coincide in some places with the
morphism $\exp_G$ as we define it below ({\em loc. cit.}
chap.~II, \S~4, 3.7) and in other places with the morphism
$\gamma_G$ ({\em loc. cit.} chap.~II, \S~4, 4.2). The drawback
of flexibility is a little loss of precision.
We introduce the exponential in a somehow more formal way,
as an actual morphism between functors.

\begin{definition} \label{def:exponential}
The {\em exponential} of a $k$-group scheme $G$ is the
composition:
\[
\xymatrix@C=12mm{
\exp_{G,I} : h^*\Lie(G,I) \ar[r]^-{h^*\gamma_G} &
h^*\T\!G \ar[r]^-{\beta_{h^*G}} & h^*G.}
\]
\end{definition}

When $I$ is clear from context, and also when $I=k\eps$,
we write $\exp_G$ instead of $\exp_{G,I}$. The following
proposition collects some elementary properties of the exponential.

\begin{prop} \label{prop:elementary properties of exp}
The exponential $\exp_{G,I}$ of a $k$-group
scheme $G$ has the following properties.
\begin{trivlist}
\itemn{1. Functoriality} For all morphisms of group functors
$f:G\to G'$ we have a commutative square:
\[
\xymatrix@C=15mm{
h^*\Lie(G,I) \ar[r]^-{\exp_{G,I}} \ar[d]_-{h^*\Lie(f)}
& h^*G \ar[d]^{h^*f} \\
h^*\Lie(G',I) \ar[r]^-{\exp_{G',I}} & h^*G'.}
\]
\itemn{2. Equivariance} The map $\exp_{G,I}$ is $h^*G$-equivariant
for the adjoint action on $h^*\Lie(G,I)$ and the conjugation
action of $h^*G$ on itself.
\itemn{3. Infinitesimal translation}
Let $x$ be a point of $h^*\Lie(G,I)$, $g$ a point of $h^*G$
and $(x,g)=h^*\gamma_G(x)\cdot h^*\alpha_G(g)$ the corresponding
point of $h^*\T(G,I)$. Then we have
$\beta_{h^*G}(x,g)=\exp_{G,I}(x)g$.
\itemn{4. Adjoint action of exponentials}
Using the description $\Lie(G,I)=\Lie G\otimes\VV(I^\smallvee)$,
the morphism
\[
h^*\!\Ad\circ\exp_{G,I}:h^*\Lie(G,I) \to h^*\GL(\Lie(G,I))
\]
is equal to $x\otimes i \mapsto \id+i\ad(x)$, that is,
$\Ad(\exp_{G,I}(x\otimes i))x'=x'+i[x,x']$.
\itemn{5. Exponential of a Lie algebra}
Let $J$ be another finite free $k$-module. Using the isomorphism
\[
\Lie(G,J)\otimes\VV(I^\smallvee)\simeq
\Lie(G,J\otimes_k I)\stackrel{d}{\isomto}\Lie(\Lie(G,J),I)
\]
from Lemma~\ref{lemma:isom d}, the morphism of $k[I]$-group schemes
\[
\xymatrix@C=20mm{
h^*\Lie(G,J)\otimes_{\OO_{k[I]}}h^*\VV(I^\smallvee) = h^*\Lie(\Lie(G,J),I)
\ar[r]^-{\exp_{\Lie(G,J),I}} & h^*\Lie(G,J)}
\]
is given by the external law of the $\OO_{k[I]}$-module
scheme $h^*\Lie(G,J)$. In particular, its image is
$I\!\cdot\!h^*\Lie(G,J)$ and its kernel contains
$I\!\cdot\!h^*\Lie(\Lie(G,J),I)$. Besides, if $I=k\eps$
the map $\exp_{\Lie(G,J)}:h^*\Lie(G,J)\to h^*\Lie(G,J)$ is
multiplication-by-$\eps$ in the module scheme $h^*\Lie(G,J)$.
\itemn{6. Kernel} The two maps:
\[
\exp_{G,I}:h^*\Lie(G,I)\too h^*G
\ ,\quad
\exp_{\Lie(G),I}:h^*\Lie(G,I)\stackrel{h^*d}{\simeq}
h^*\Lie(\Lie(G),I) \too h^*\Lie(G)
\]
have equal kernels, thus
$I\!\cdot\!h^*\Lie(G,I)\subset\ker(\exp_{G,I})$.
In particular, in case $I=k\eps$, the kernel of the morphism
$\exp_G:h^*\Lie G\to h^*G$ is equal to the kernel of
the multiplication-by-$\eps$ map in $h^*\Lie G$. \hfill $\square$
\end{trivlist}
\end{prop}

We finish this subsection with a corollary of the computation
of the exponential of a Lie algebra.

\begin{lemma} \label{lemma:cocycle vanishes with eps}
Let $G,H$ be two group schemes over $k$. Let $\phi:G\to \Lie(H,I)$
be a morphism of pointed schemes. Let $i$ be a section of
the ideal $I\!\cdot\!\OO_{k[I]}\subset \OO_{k[I]}$. Then the
following compositions are both equal to the trivial morphism:
\begin{trivlist}
\itemn{1}
$\xymatrix@C=12mm{
h^*\Lie(G,I)\ar[r]^-{\exp_{G,I}} & h^*G \ar[r]^-{h^*\phi}
& h^*\Lie(H,I) \ar[r]^-{i} & h^*\Lie(H,I)}$.
\itemn{2}
$\xymatrix@C=12mm{
h^*\Lie(G,I)\ar[r]^-{\exp_{G,I}} & h^*G \ar[r]^-{h^*\phi}
& h^*\Lie(H,I) \ar[r]^-{\exp_{H,I}} & h^* H}$.
\end{trivlist}
\end{lemma}

\begin{proof}
By functoriality of $\exp$ we have a commutative diagram:
\[
\xymatrix@C=25mm{
h^*\Lie(G,I) \ar^{\exp_{G,I}}[r] \ar_{h^*\diff\!\phi}[d]
& h^*G\ar^{h^*\!\phi}[d] \\
h^*\Lie(\Lie(H,I),I) \ar^-{\exp_{\Lie(H,I),I}}[r]
& h^*\Lie(H,I).}
\]
According to
Proposition~\ref{prop:elementary properties of exp}(5),
the image of $\exp_{\Lie(H,I),I}$ is equal to the subfunctor
$I\cdot h^*\Lie(G,I)$. Since $I^2=0$, from this (1) follows.
Moreover, by~\ref{prop:elementary properties of exp}(6)
the map $\exp_{H,I}$ has the same kernel as
$\exp_{\Lie(H),I}$ which
by~\ref{prop:elementary properties of exp}(5) contains
$I\!\cdot\!h^*\Lie(H,I)$. Again, since $I^2=0$, point (2) follows.
\hfill $\Box$
\end{proof}

\subsection{Deformations of affine group schemes}
\label{subsection:defos}

In this subsection, we illustrate the usefulness of the group
algebra in two ways. First, in
Proposition~\ref{prop:embeddings in OG}
we show how concepts of differential calculus can be handled
very conveniently using the group algebra. We include the
examples of the adjoint action and the Lie bracket. The
results of Proposition~\ref{prop:elementary properties of exp}
can be derived painlessly in a similar fashion. Then
in Proposition~\ref{rigid defos of affine groups} we show how
to recover directly the fact that isomorphism
classes of deformations over $k[I]$ of a smooth, affine $k$-group
scheme $G$ are classified by the second cohomology group
$\mathrm{H}^2(G,\Lie(G,I))$.

\begin{prop} \label{prop:embeddings in OG}
Let $G$ be an affine $k$-group scheme and $(\OO_k[G],+,\star)$
its group algebra.
\begin{trivlist}
\itemn{1} The tangent bundle $\T(\OO_k[G],I)$ of the group algebra $\OO_k[G]$ is canonically isomorphic to the algebra scheme $\OO_k[G][I]=\OO_k[G]\oplus \OO_k[G]\!\cdot\!I$.
We have a commutative diagram of affine monoid schemes:
\[
\xymatrix@C=5mm{
1 \ar[r] & \underset{}{\Lie G\otimes\VV(I^\smallvee)}
\ar[rr] \ar@{^(->}[d]
& & \underset{}{\T(G,I)} \ar[rr] \ar@{^(->}[d]
& & \underset{}{G} \ar[r] \ar@{^(->}[d] & 1 \\
1 \ar[r] & \underset{}{\overset{}{(\OO_k[G]\!\cdot\!I,+)}}
\ar[rr] \ar@{=}[d]
& & \underset{}{\overset{}{(\OO_k[G][I]^\times,\star)}}
\ar[rr] \ar@{^(->}[d]
& & \underset{}{\overset{}{(\OO_k[G]^\times,\star)}}
\ar[r] \ar@{^(->}[d] & 1 \\
& \overset{}{(\OO_k[G]\!\cdot\!I,+)} \ar[rr]^-{a_2\mapsto 1+a_2}
& & \overset{}{(\OO_k[G][I],\star)} \ar[rr]^-{a_1+a_2\mapsto a_1}
& & \overset{}{(\OO_k[G],\star)} & }
\]
The first two rows are split exact sequences of group
schemes. In the last row we have written the points of
$\OO_k[G][I]$ as $a=a_1+a_2$ with $a_1\in\OO_k[G]$ and
$a_2 \in \OO_k[G]\!\cdot\!I$.
The map
$\Lie G\otimes\VV(I^\smallvee)\into (\OO_k[G]\!\cdot\!I,+)$ induces an
isomorphism between $\Lie G\otimes\VV(I^\smallvee)$ and the subscheme
$\Der_G^1\subset\OO_k[G]$ of
$e$-derivations where $e$ is the neutral element of $G$.
\itemn{2} The adjoint action $\Ad:G\to\GL(\Lie(G,I))$
can be expressed as a conjugation inside $\OO_k[G][I]$:
\[
1+\Ad(g)x=g(1+x)g^{-1}=1+gxg^{-1},
\]
for all $x\in \Lie(G,I)=\Lie G\otimes\VV(I^\smallvee)$ and $g\in G$.
\itemn{3} The Lie bracket $[-,-]:\Lie G\times\Lie G\to \Lie G$
can be expressed as the bracket
defined by the associative multiplication of $\OO_k[G]$:
\[
[x,y]=xy-yx.
\]
\end{trivlist}
\end{prop}

\begin{remark}
More generally, for all finite free $k$-modules $I,J$
we have an infinitesimal adjoint action
\[
\ad:=\Lie(\Ad,J):\Lie(G,J)\to\Hom(\Lie(G,I),\Lie(G,I\otimes J))
\]
which can be expressed as the bracket defined by the
associative multiplication of $\OO_k[G][I][J]$, that is,
$[x,y]=xy-yx$ for $x\in \OO_k[G]\!\cdot\!J$ and
$y\in\OO_k[G]\!\cdot\!I$. The expression in the particular
case $I=J=k\eps$ is obtained by writing $1+\eps x$ instead
of $1+x$.
\end{remark}

\begin{proof}
(1) By Proposition~\ref{prop:group algebra}(4),
we have an isomorphism of $\OO_k$-algebra schemes:
\[
\T(\OO_k[G],I)=h_*h^*\OO_k[G] \isomto
\OO_k[G]\otimes_{\OO_k}h_*\OO_{k[I]}
\simeq\OO_k[G][I].
\]
Under this identification, the map
$\T\nu_G:\T(G,I)\to\T(\OO_k[G],I)=\OO_k[G][I]$ can be described as follows:
for each $k$-algebra~$R$, a point $f\in\T(G,I)(R)$ is a morphism
$f:A\to R[I]$, $f(x)=u(x)+v(x)$ for some unique $k$-module
homomorphisms $u:A\to R$ and $v:A\to IR$, and we have
\[
\T\nu_G(f)=u+v \in \OO_k[G][I](R).
\]
This is the central vertical map in the pictured diagram.
The rest is clear.

\smallskip

\noindent (2) Using the inclusions of multiplicative monoids
$\alpha_G:G\into \T\!G$ and $\T\nu_G:\T\!G\into \OO_k[G][\eps]$,
we can view the conjugation action by $G$ inside the tangent
bundle or inside the tangent group algebra, as we wish.
The result follows.

\smallskip

\noindent (3) In order to compute $\ad$ we
differentiate and hence work in $\OO_k[G][\eps,\eps']$. That is,
the Lie algebra embedded by $y\mapsto 1+\eps y$ is acted upon by
the Lie algebra embedded by $x\mapsto 1+\eps 'x$, via conjugation
in the ambient $\OO_k[G][\eps,\eps']$. With these notations, the
identification $\End(\Lie G)\isomto \Lie(\GL(\Lie G))$ goes by
$f\mapsto 1+\eps' f$. All in all, the outcome is that $\ad(x)$
is determined by the condition that for all $y$ we have:
\[
(1+\eps' x)(1+\eps y)(1-\eps' x)=
1+\eps\big((\id+\eps'\ad(x))(y)\big).
\]
Since the left-hand side is equal to
$1+\eps y+\eps\eps'(xy-yx)$, this proves our claim.
\hfill $\Box$
\end{proof}

\begin{prop} \label{rigid defos of affine groups}
Let $k$ be a base ring and let $G$ be an affine $k$-group scheme.
\begin{trivlist}
\itemn{1} The set of $k[I]$-group scheme structures on the
scheme $h^*G$ that lift the $k$-group scheme structure of $G$
is in bijection with the set of 2-cocycles $c:G\times G\to \Lie G$.
\itemn{2} The set of isomorphism classes of rigid
deformations of $G$ over $k[I]$ is in bijection with
$\mathrm{H}^2(G,\Lie(G,I))$,
the second group cohomology of $G$ with coefficients in the adjoint
representation $\Lie(G,I)\simeq\Lie G\otimes\VV(I^\smallvee)$.
\end{trivlist}
\end{prop}

\begin{proof}
(1) We want to deform the multiplication
$m:G\times G\too G, \quad (u,v)\mapsto uv$ into a multiplication
$\tilde m:h^*G\times h^*G\to h^*G$ which by adjunction
we can view as a map:
\[
G\times G\to h_*h^*G=\T(G,I),
\quad (u,v)\mapsto u\odot v.
\]
We use embeddings inside the group algebra as
in Proposition~\ref{prop:embeddings in OG}; thus both
targets of $m$ and $\tilde m$ are embedded in $\OO_k[G][I]$.
The condition that $\tilde m$ equals $m$ modulo $I$
is that these morphisms agree after composition with the
projection $\pi:\OO_k[G][I]\to\OO_k[G]$. In other words the
condition is that for all points $u,v\in G$, the element
$(u\odot v)(uv)^{-1}$ equals 1 modulo $I$. Since this is also
a point of $\T(G,I)$, it belongs to $\Lie(G,I)$. Hence we can write
\[
u\odot v=(1+c(u,v))uv
\]
for some $c:G\times G\to\Lie(G,I)$. The associativity constraint
$(u\odot v)\odot w=u\odot (v\odot w)$ gives
the cocycle relation:
\[
c(uv,w)+c(u,v)=c(u,vw)+uc(v,w)u^{-1}.
\]
Conversely, if $c:G\times G\to \Lie(G,I)$ is a 2-cocycle,
we define
\[
u\odot v:=(1+c(u,v))uv.
\]
The cocycle identity gives the associativity of this map.
Moreover, inverses for this law exist and are given by the formula
$u^{\odot -1}=u^{-1}(1-c(u,u^{-1}))$. So we obtain a
$k[I]$-group scheme $\mathscr{G}_c:=(\Spec(A[I]),\odot)$
which is a deformation of $G$.

\smallskip

\noindent (2) Let $\mathscr{G}$ be a rigid deformation of $G$. Choosing an
isomorphism of schemes $\varphi_1:\mathscr{G}\isomto h^*G$, the induced
group scheme structure on $h^*G$ gives rise to a cocycle $c$
as explained in (1). Choosing another isomorphism
$\varphi_2:\mathscr{G}\isomto h^*G$, we have an automorphism
$\xi=\varphi_2\circ\varphi_1^{-1}:h^*G\to h^*G$
which restricts to the identity on $G$. The map
$G\to h_*h^*=\T(G,I)$ obtained by adjunction is of the form
$u\mapsto (1+\psi(u))u$ for some morphism $\psi:G\to\Lie(G,I)$.
This means that $\xi(u)=(1+\psi(u))u$.
We want to see how the multiplication is transformed by the change
of isomorphism. If $u'=(1+\psi(u))u$ then $u=(1-\psi(u'))u'$,
a short computation shows that that the map
$\xi\circ\odot\circ\xi^{-1}$ takes $(u',v')$ to
$[1+(\psi(u'v')+c(u',v')-\psi(u')-\Ad(u')\psi(v'))]u'v'$.
We see that $c$ changes by the coboundary $\partial\psi$
and the class $[c] \in \mathrm{H}^2(G,\Lie(G,I))$ does not depend on the
choice of isomorphism $\mathscr{G}\isomto h^*G$. To obtain the inverse
bijection, one chooses a cocycle $c$ and attaches the deformation
$\mathscr{G}_c$ as in (1). The map is well-defined because another
choice of $c$ in the same cohomology class gives an isomorphic
deformation.
\hfill $\Box$
\end{proof}

\appendix{B}{Module stacks in groupoids}
\label{appendix:module groupoids}

Both categories $\Gr\!/k[I]$ and $\Ext(I)/k$ are endowed
with the structure of {\em $\OO_k$-module stacks in
groupoids over $\Gr\!/k$} and the purpose of this Appendix is
to explain what this means. In the two cases this seems to be a
well-known fact, but we were able to locate only very few
discussions of this topic in the literature. In fact, the
additive part of the structure, which goes by the name of a
``Picard category'', is well documented, a landmark being
Deligne's expos\'e in~\cite{SGA4.3}, Exp.~XVIII, \S~1.4.
However, the linear part of the structure, that is the
$\OO_k$-scalar multiplication and its interplay with the
additive structure, is almost absent from the literature.
Subsections 2.3, 2.4
and 2.5 of Osserman~\cite{Os10} are a first step, but the
author writes: {\em Although it is possible to 
{\em [define scalar multiplication maps]} on a categorical
level as we did with addition, expressing the proper conditions
for associativity and distributivity isomorphisms becomes
substantially more complicated}. Here we simply provide a
definition in due form. An extended version of the article
available on the authors' webpage contains a treatment including
a few basic results to highlight the nontrivial features of
the theory.

We start with the definition of Picard categories. The
alternative phrase {\em commutative group groupoids} is a
more accurate name to refer to them, but we stick with the
traditional name.
In order to make the axioms reader-friendly, we adopt a
simplified description for the multifunctors involved, e.g.
the associativity isomorphism $a:T_1\to T_2$ between
the trifunctors $T_1,T_2:P\times P\times P\to P$ given by
$T_1=+\circ (+\times \id)$ and $T_2=+\circ (\id\times +)$
is given in the form of isomorphisms
$a_{x,y,z}:(x+y)+z\to x+(y+z)$ functorial in $x,y,z\in P$.

\begin{definition*} \label{definition:constraints}
Let $P$ be a category and $+:P\times P\to P$ a bifunctor.
\begin{trivlist}
\itemn{1} An {\em associativity constraint} for $+$ is an
isomorphism of functors $a_{x,y,z}:(x+y)+z\isomto x+(y+z)$
such that the pentagon axiom (\cite{SGA4.3}, Exp.~XVIII, 1.4.1)
is satisfied. It is called {\em trivial} or {\em strict}
if $a_{x,y,z}=\id$ for all $x,y,z\in P$.
\itemn{2} A {\em commutativity constraint} for $+$ is an
isomorphism of functors $c_{x,y}:x+y\isomto y+x$ which
satisfies $c_{y,x}\circ c_{x,y}=\id_{x+y}$ for all $x,y\in P$.
It is called {\em trivial} or {\em strict} if $c_{x,y}=\id$
for all $x,y\in P$.
\itemn{3} The associativity and commutativity constraints $a$
and $c$ are {\em compatible} if the hexagon axiom
(\cite{SGA4.3}, Exp.~XVIII, 1.4.1) is satisfied.
\itemn{4} A {\em neutral element} for $+$ is an object
$0\in P$ with an isomorphism $\varphi:0+0\isomto 0$.
\end{trivlist}
\end{definition*}

In other mathematical contexts, associativity constraints
are called {\em associators} and commutativity
constraints are called {\em symmetric braidings}.

\begin{definition*}
Let $(P_1,+)$ and $(P_2,+)$ be categories endowed with
bifunctors. Let $F:P_1\to P_2$ be a functor and
$\varphi_{F,x,y}:F(x+y)\isomto F(x)+F(y)$ an
isomorphism of functors.
\begin{trivlist}
\itemn{1} Let $a_1,a_2$ be associativity constraints on
$(P_1,+)$ and $(P_2,+)$. We say that $(F,\varphi_F)$
is {\em compatible with $a_1,a_2$} if the following diagram
commutes:
\[
\xymatrix@C=15mm{
F((x+y)+z) \ar[r]^-{\varphi_F} \ar[d]_{F(a_1)}
& F(x+y)+F(z) \ar[r]^-{\varphi_F}
& (F(x)+F(y))+F(z) \ar[d]^{a_2} \\
F(x+(y+z)) \ar[r]^-{\varphi_F}
& F(x)+F(y+z) \ar[r]^-{\varphi_F} & F(x)+(F(y)+F(z)).}
\]
\itemn{2} Let $c_1,c_2$ be commutativity constraints on
$(P_1,+)$ and $(P_2,+)$. We say that $(F,\varphi_F)$ is
{\em compatible with $c_1,c_2$} if the following diagram
commutes:
\[
\xymatrix@C=15mm{
F(x+y) \ar[r]^-{\varphi_F} \ar[d]_-{F(c_1)}
& F(x)+F(y) \ar[d]^-{c_2} \\
F(y+x) \ar[r]^-{\varphi_F} & F(y)+F(x).}
\]
\end{trivlist}
\end{definition*}

\begin{definition*}
A {\em Picard category} is a quadruple $(P,+,a,c)$ composed
of a nonempty groupoid $P$, a bifunctor $+:P\times P\to P$
with compatible associativity and commutativity constraints
$a$ and $c$, such that for each $x\in P$
the functor $P\to P$, $y\mapsto x+y$ is an equivalence.
\end{definition*}

\bigskip

Any Picard category $P$ has a neutral element $0$
which is unique up to a unique isomorphism
(\cite{SGA4.3}, Exp.~XVIII, 1.4.4). Moreover, for each
$x,y\in P$ the set of morphisms $\Hom(x,y)$ is either empty
or a torsor under the group $G:=\Aut(0)$. More precisely, the
functors $+x:P\to P$ and $x+:P\to P$ induce the same
bijection $G\to\Aut(x)$, $\varphi\mapsto\varphi+\id_x$.
Viewing this bijection as an identification, the set
$\Hom(x,y)$ with its right $\Aut(x)$-action and left
$\Aut(y)$-action becomes a pseudo-$G$-bitorsor, i.e.
it is either empty or a $G$-bitorsor.

\begin{definition*}
Let $P_1,P_2$ be Picard categories.
\begin{trivlist}
\itemn{1} An {\em additive functor}
is a pair $(F,\varphi_F)$ where $F:P_1\to P_2$ is a
functor and $\varphi_{F,x,y}:F(x+y)\isomto F(x)+F(y)$ is an
isomorphism of functors that is compatible with associativity
and commutativity constraints.
\itemn{2}
Let $F,G:P_1\to P_2$ be additive functors.
A {\em morphism of additive functors} is a morphism of functors
$u:F\to G$ such that the following diagram is commutative:
\[
\xymatrix@C=15mm{
F(x+y) \ar[r]^-{u_{x+y}} \ar[d]_{\varphi_F}
& G(x+y) \ar[d]^{\varphi_G} \\
F(x)+F(y) \ar[r]^-{u_x+u_y} & G(x)+G(y).}
\]
\end{trivlist}
\end{definition*}

We emphasize that since a Picard category is a groupoid
(that is, all its morphisms are isomorphisms), all
morphisms of additive functors $u:F\to G$ are isomorphisms.

The {\em category of additive functors} $\Hom(P_1,P_2)$
is itself a Picard category
(\cite{SGA4.3}, Exp.~XVIII, 1.4.7). Additive functors can be
composed and the identity functors behave as neutral elements.
In the particular case where $P_1=P_2=P$, along with its
addition law, the Picard category $\End(P)=\Hom(P,P)$ enjoys
an internal multiplication given by composition. Note that
in this case multiplication is strictly associative,
because so is composition of functors in categories.

In fact $\End(P)$ is a {\em ring category}, but in order
to introduce module groupoids, we do not
actually need to define what is such a thing.

\begin{definition*} \label{definition:module groupoid}
Let $\Lambda$ be a commutative ring.
A {\em $\Lambda$-module groupoid} is a Picard category $P$
endowed with a functor $F=(F,\varphi_F,\psi_F):\Lambda\to\End(P)$
called {\em scalar multiplication} such that:
\begin{trivlist}
\itemm{1} $(F,\varphi_F)$ is an additive functor.
\end{trivlist}
For each $\lambda\in \Lambda$, for simplicity we write
$(\lambda,\varphi_\lambda)$ for
$(F_\lambda,\varphi_{F_\lambda}):P\to P$. Moreover:
\begin{trivlist}
\itemm{2} $(F,\psi_F)$ is multiplicative, i.e. $F(1)=\id_P$
and $F$ is compatible with the associativity constraints of
multiplication.
\itemm{3} $F$ is compatible with the distributivity of
multiplication over addition:
\[
\xymatrix@C=15mm{
(\lambda(\mu+\nu))x \ar[r]^-{\psi_F} \ar@{=}[d] &
\lambda((\mu+\nu)x) \ar[r]^-{\varphi_F}
& \lambda(\mu x+\nu x) \ar[d]^-{\varphi_\lambda} \\
(\lambda\mu+\lambda\nu)x \ar[r]^-{\varphi_F} &
(\lambda\mu)x+(\lambda\nu)x \ar[r]^-{\psi_F+\psi_F} &
\lambda(\mu x)+\lambda(\nu x)}
\]
commutes.
\end{trivlist}
\end{definition*}

\begin{definition*} \label{definition:module stack}
Let $\cS$ be a site. Let $\Lambda$ be a sheaf of commutative
rings on $\cS$. A {\em $\Lambda$-module stack (in groupoids)
over $\cS$} is a stack in groupoids $P$ over $\cS$
endowed with
\begin{trivlist}
\itemm{1} a functor $+:P\times P\to P$,
\itemm{2} isomorphisms of functors
$a_{x,y,z}:(x+y)+z\isomto x+(y+z)$ and
$c_{x,y}:x+y\isomto y+x$,
\itemm{3} a functor $F=(F,\varphi_F,\psi_F):\Lambda\to\End(P)$,
\end{trivlist}
such that for each $U\in\cS$ the fibre category $P(U)$
is a $\Lambda(U)$-module groupoid.
\end{definition*}

\bigskip

There is an obvious corresponding relative notion of
{\em $\Lambda$-module stack (in groupoids)} over a given
$\cS$-stack $Q$, namely, it is a morphism of stacks $P\to Q$
that makes $P$ a stack fibred in groupoids over $Q$,
with an addition functor $+:P\times_Q P\to P$ etc.
It is the relative notion that is useful in the paper.

\providecommand{\bysame}{\leavevmode\hbox to3em{\hrulefill}\thinspace}
%
%

\bibliographystyle{amsalpha}
\bibliographymark{References}
\def\cprime{$'$}

\end{document}